\documentclass[10pt,a4paper]{article}

\usepackage{booktabs}
\usepackage{amssymb,latexsym}
\usepackage{a4}
\usepackage[all]{xy}

\newif\iflabel
\labelfalse
\newcommand{\Label}[1]{\iflabel\ifmmode\makebox[0pt][l]{[#1]}
	       \else\marginpar{[#1]}
	       \fi\fi\label{#1} }

\newcommand{\ra}{\rangle}
\newcommand{\la}{\langle}
\newcommand{\id}{\mathrm{id}}
\newcommand{\X}{{\mathcal X}}
\newcommand{\Q}{{\mathcal Q}}
\newcommand{\R}{{\mathcal R}}
\renewcommand{\S}{{\mathcal S}}
\newcommand{\Z}{{\mathbb Z}}
\newcommand{\N}{{\mathbb N}}

\newcommand{\End}{{\mathrm{End}}}
\newcommand{\mc}{\mathcal}
\newcommand{\im}{{\mathrm{im}}}

\newcommand{\Gap}{{\scshape Gap}}

\newcommand{\Grig}{{\mathfrak G}}

\newcommand{\ti}{\tilde}

\newcommand{\Rom}[1]{\MakeUppercase{\romannumeral #1}}
\newcommand{\mod}{\:{\mathrm{mod}}\:}

\newenvironment{proof}{\par\vskip-\lastskip\vskip\topsep
\noindent{\it Proof.}\vadjust{\nobreak}\quad
\begingroup\divide\topsep3\divide\itemsep3
\divide\partopsep3\divide\parskip3
\divide\parsep3}
{\ifvmode\penalty10000\hbox to\hsize{\hfil$\Box$}
\else\parfillskip0pt\widowpenalty10000\hfil$\Box$
\fi\par\vskip 1.5ex\endgroup}

\newenvironment{CenProof}[1]{\par\vskip-\lastskip\vskip\topsep
\noindent{\it Proof#1.}\vadjust{\nobreak}\quad
\begingroup\divide\topsep3\divide\itemsep3
\divide\partopsep3\divide\parskip3
\divide\parsep3}
{\ifvmode\penalty10000\hbox to\hsize{\hfil$\Box$}
\else\parfillskip0pt\widowpenalty10000\hfil$\Box$
\fi\par\vskip 1.5ex\endgroup}

\newtheorem{theorem}{Theorem}[section]
\newtheorem{corollary}[theorem]{Corollary}
\newtheorem{lemma}[theorem]{Lemma}

\newtheorem{proposition}[theorem]{Proposition}
\newtheorem{remark}[theorem]{Remark}
\newtheorem{definition}[theorem]{Definition}

\title{A Note on Invariantly Finitely $L$-Presented Groups}
\author{Ren\'e Hartung}
\date{April 2012}

\begin{document}
\maketitle

\begin{abstract}
  In the first part of this note, we introduce Tietze transformations
  for $L$-presentations.  These transformations enable us to generalize
  Tietze's theorem for finitely presented groups to invariantly finitely
  $L$-presented groups. Moreover, they allow us to prove that `being
  invariantly finitely $L$-presented' is an abstract property of a group
  which does not depend on the generating set.\smallskip

  In the second part of this note, we consider finitely generated normal
  subgroups of finitely presented groups. Benli proved that a finitely
  generated normal subgroup of a finitely presented group is invariantly
  finitely $L$-presented whenever its quotient is infinite cyclic. We
  generalize this result to the case where the finitely presented group
  splits over its finitely generated subgroup and to the case where the
  quotient is abelian with torsion-free rank at most two.\medskip

  \bigskip

\end{abstract}

\section{Introduction}\Label{sec:NTIntro}
\renewcommand{\thetheorem}{\Alph{theorem}}
Finite $L$-presentations are possibly infinite group presentations with
finitely many generators whose relations (up to finitely many exceptions)
are obtained by iteratively applying finitely many substitutions to a
finite set of relations; see~\cite{Bar03} or Section~\ref{sec:NTPrel}
for a definition. Various infinitely presented groups can be
described by a finite $L$-presentation. For example, the Grigorchuk
group~\cite{Gri80} and the Gupta-Sidki group~\cite{GS83} are finitely
$L$-presented~\cite{Lys85,Sid87,Bar03,BEH08}.  An $L$-presentation is
\emph{invariant} if the substitutions, which generate the relations, induce
endomorphisms of the group. In fact, invariant finite $L$-presentations
are finite presentations in the universe of \emph{groups with
operators}~\cite{Kr25,Noe29} in the sense that the operator domain of
the group generates the possibly infinitely many relations out of a
finite set of relations. The finite $L$-presentation for the Grigorchuk
group in~\cite{Lys85} is an example of an invariant finite
$L$-presentation~\cite{Gri98}.\smallskip

Finite $L$-presentations allow computer algorithms to be applied in the
investigation of the groups they define. For instance, they allow one to
compute the lower central series quotients~\cite{BEH08}, to compute the
Dwyer quotients of the group's Schur multiplier~\cite{Har10}, to develop
a coset enumerator for finite index subgroups~\cite{Har11}, and even the
Reidemeister-Schreier theorem for finitely presented groups generalizes
to finitely $L$-presented groups~\cite{Har11b}. For a survey on the
application of computers in the investigation of finitely $L$-presented
groups, we refer to~\cite{Har12b}.\smallskip

In the first part of this note, we introduce Tietze transformations for 
$L$-pre\-sen\-ta\-tions. These transformations allow us to generalize
Tietze's theorem for finitely presented groups~\cite{Tie08} to invariantly
finitely $L$-presented groups:
\begin{theorem}\Label{thm:LTietzeInv}
  Two invariant finite $L$-presentations define isomorphic groups if
  and only if it is possible to pass from one $L$-presentation to the
  other by a finite sequence of transformations.
\end{theorem}
If a group admits a finite presentation with respect to one
generating set, then so it does 
with respect to any other finite generating
set~\cite[Chapter~\Rom{5}]{Har00}. This result for finitely
presented groups also generalizes to invariant finite $L$-presentations:
\def\0{Bartholdi~\cite{Bar03}}
\begin{theorem}[\0]\Label{thm:AbstProp}
  Being invariantly finitely $L$-presented is an abstract property of
  a group which does not depend on the generating set.
\end{theorem}
Our proof of Theorem~\ref{thm:AbstProp} fills a gap in the proof
of~\cite[Proposition~2.2]{Bar03} because the transformations in the latter
proof are not sufficient; see Section~\ref{sec:NTAppsTie} below.\smallskip

In the second part of this note, in Section~\ref{sec:FGNorOfFP}, we
consider finitely generated normal subgroups of finitely presented
groups. By Higman's embedding theorem, every finitely generated group
embeds into a finitely presented group if and only if it is recursively
presented~\cite{Hig61}. Since every finite $L$-presentation is recursive,
finitely $L$-presented groups therefore embed into finitely presented
groups. As indicated in~\cite{Ben11}, we prove that every group which
admits an invariant finite $L$-presentation, where each substitution
induces an automorphism of the group, embeds as a \emph{normal}
subgroup into a finitely presented group. On the other hand,
the Reidemeister-Schreier theorem for finitely $L$-presented groups
in~\cite{Har11b} shows that every normal subgroup of a finitely presented
group admits an invariant $L$-presentation where each substitution
induces an automorphism of the group; the obtained $L$-presentation is
finite if and only if the normal subgroup has finite index.\smallskip

Finitely generated \emph{normal} subgroups of finitely presented groups
with infinite index were considered in~\cite{Ben11}: It was proved that
a finitely generated normal subgroup of a finitely presented group
is invariantly finitely $L$-presented if its quotient is infinite
cyclic. Moreover, in~\cite[Remark~(2)]{Ben11}, Benli asked for a
generalization of his latter result and he posed the following problem:
\begin{quote}
  {\it
  Is it true that a finitely generated group embeds as a normal subgroup
  into a finitely presented group if and only if it admits an invariant
  finite $L$-presentation where each substitution induces an automorphism
  of the group?}
\end{quote}
We generalize Benli's constructions from~\cite{Ben11} in order to prove the
following 
\begin{theorem}\Label{thm:GSplits}
  Every finitely generated normal subgroup of a finitely presented group
  is invariantly finitely $L$-presented if the group splits over
  its subgroup.
\end{theorem}
Since $G$ splits over its subgroup $H \unlhd G$
if $G / H$ is a free group, Benli's result in~\cite{Ben11} is a
consequence of Theorem~\ref{thm:GSplits}. Moreover, our generalizations
of the constructions from~\cite{Ben11} allows us to prove
\begin{theorem}\Label{thm:TorRk2}
  Every finitely generated normal subgroup of a finitely presented
  group is invariantly finitely $L$-presented whenever the quotient is
  abelian with torsion-free rank at most two.
\end{theorem}
Our constructions do not generalize further; see
Remark~\ref{rem:NTNonInvLp}.

\renewcommand{\thetheorem}{\arabic{section}.\arabic{theorem}}
\setcounter{theorem}{0}

\section{Preliminaries}\Label{sec:NTPrel}
In this section, we recall the notion of an invariant
finite $L$-presentation as introduced in~\cite{Bar03}.
An \emph{$L$-presentation} is a group presentation of the form
\begin{equation}
  \Big\la \X\,\Big|\, \Q \cup \bigcup_{\sigma\in\Phi^*} \R^\sigma \Big\ra,
  \Label{eqn:NTLpres}
\end{equation}
where $\X$ is an alphabet, $\Q$ and $\R$ are subsets of the free group
$F = F(\X)$ over the alphabet $\X$, and $\Phi^* \subseteq \End(F)$
denotes the monoid of endomorphisms that is generated by $\Phi$. If
the \emph{generators} $\X$, the \emph{fixed relations} $\Q$, the
\emph{substitutions} $\Phi$, and the \emph{iterated relations} $\R$ have
finite cardinality, the $L$-presentation in Eq.~(\ref{eqn:NTLpres})
is a \emph{finite $L$-presentation}. We also write $\la \X\mid\Q
\mid\Phi\mid\R\ra$ for the $L$-presentation in Eq.~(\ref{eqn:NTLpres}) and
$G = \la \X\mid\Q\mid\Phi\mid\R\ra$ for the group it defines.\smallskip

A group which admits a finite $L$-presentation is \emph{finitely
$L$-presented}. An $L$-pre\-sen\-ta\-tion of the form
$\la\X\mid\emptyset\mid\Phi\mid\R\ra$ is \emph{ascending} and an
$L$-presentation \mbox{$\la\X\mid\Q\mid\Phi\mid\R\ra$} is called
\emph{invariant} (and the group it defines is \emph{invariantly
$L$-pre\-sented}), if each substitution $\varphi \in \Phi$ induces
an endomorphism of the group; i.e., if the normal subgroup $\la
\Q \cup \bigcup_{\sigma\in\Phi^*} \R^\sigma \ra^F \unlhd F$ is
$\varphi$-invariant. Each ascending $L$-presentation is invariant
and each invariant $L$-presentation $\la \X\mid\Q\mid\Phi\mid\R\ra$
admits an ascending $L$-presentation \mbox{$\la \X\mid\emptyset\mid
\Phi\mid\Q\cup\R\ra$} which defines the same group; see
Proposition~\ref{prop:NTReplRelsInvLp}. Even though invariant and
ascending $L$-presentations are essentially the same, we like to
distinguish between these two objects. The finite $L$-presentation
in~\cite{Lys85} for the group constructed by Grigorchuk~\cite{Gri80}
is not ascending but it is easy to see that it is an invariant
$L$-presentation; see, for instance,~\cite[Corollary~4]{Gri98}.
\begin{remark}\Label{rem:NTNonInvLp}
  There are finite $L$-presentations that are not invariant.
\end{remark}
\begin{proof}
  The free product $\Z_2 * \Z_2 = \la\{a,b\}\mid\{a^2, b^2\}\ra$ is
  finitely $L$-presented by\linebreak $\la \{a,b\} \mid \{a^2\} \mid \{\sigma\}
  \mid \{b^2\} \ra$ where $\sigma$ is induced by the map $a \mapsto ab$
  and $b \mapsto b^2$. If this $L$-presentation were invariant, the
  ascending $L$-presentation \mbox{$\la \{a,b\}\mid \emptyset\mid
  \{\sigma\} \mid \{a^2,b^2\}\ra$} would also define $\Z_2 *
  \Z_2$; see Proposition~\ref{prop:NTReplRelsInvLp}. In this case
  $(a^2)^\sigma = abab$ is a relation in the group and, since $a^2 =
  b^2 = 1$ holds, the generators $a$ and $b$ commute. Therefore the
  ascending $L$-presentation defines a quotient of the $2$-elementary
  abelian group $\Z_2 \times \Z_2$. In fact, it defines a finite
  group. Thus $\la\{a,b\}\mid\emptyset\mid\{\sigma\}\mid\{a^2,b^2\}\ra$
  is not a finite $L$-presentation for $\Z_2 * \Z_2$ and hence
  $\la\{a,b\}\mid\{a^2\}\mid\{\sigma^2\}\mid\{b^2\}\ra$ is not an
  invariant $L$-presentation.
\end{proof}
Note that this latter proof from~\cite{Har11b} provides
a `method' to prove that a finite $L$-presentation
$\la\X\mid\Q\mid\Phi\mid\R\ra$ is invariant; namely, if the ascending
$L$-presentation $\la\X\mid\emptyset\mid\Phi\mid\R\cup\Q\ra$ defines a
group which is isomorphic to the first. In general, we are not aware of a
method which allows us to decide whether or not a finite $L$-presentation
is invariant --- even if we assume that the $L$-presented group has a
solvable word problem.\smallskip

Invariant finite $L$-presentations are `natural' generalizations
of finite presentations because every finitely presented
group $\la\X\mid\R\ra$ is invariantly finitely $L$-presented
by $\la\X\mid\emptyset\mid\emptyset\mid\R\ra$. However, invariant
finite $L$-presentations have been used to describe various examples of
self-similar groups that are not finitely presented~\cite{Lys85,BS10}. For
instance, the group $\Grig$ constructed by Grigorchuk in~\cite{Gri80}
is not finitely presented~\cite{Gri99} but it is invariantly finitely
$L$-presented, see also~\cite{Gri98}:
\def\0{Lys{\"e}nok~\cite{Lys85}}
\begin{theorem}[\0]\Label{thm:Lyseniok}
  The Grigorchuk group is invariantly finitely $L$-pre\-sented by
  $\left\la \{a,b,c,d\} \mid \{a^2,b^2,c^2,d^2,bcd\} \mid
        \{\sigma\}\mid \{ (ad)^4,(adacac)^4 \} \right\ra$
  where $\sigma$ denotes the endomorphism of the free group over
  $\{a,b,c,d\}$ that is induced by the map $a\mapsto aca$, $b\mapsto d$,
  $c\mapsto b$, and $d\mapsto c$.
\end{theorem}
It is easy to see (and it follows with our Tietze transformations
below) that the group $\Grig$ is also invariantly finitely
$L$-presented by
\begin{equation}\Label{eqn:GrigACD}
  \Grig\cong \left\la \{a,c,d\} \mid \{a^2,c^2,d^2,(cd)^2\} \mid
        \{\ti\sigma\}\mid \{ (ad)^4,(adacac)^4 \} \right\ra,
\end{equation}
where $\ti\sigma$ is induced by the map $a\mapsto aca$, $c\mapsto cd$,
and $d\mapsto c$. Further examples of invariantly finitely $L$-presented 
groups arise, for instance, as certain wreath-products: In
contrast to~\cite{Bar03}, Bartholdi noticed that the lamplighter group
$\Z_2 \wr \Z$ is invariantly finitely $L$-presented by
\[
  \left\la \{a,t\}~\middle|~ \emptyset~\middle|~ \{ \delta \}~\middle|~
  \{a^2, [a,a^t] \}\right\ra,
\]
where $\delta$ is induced by the map $a \mapsto a^ta$ and $t \mapsto t$.
This recent result generalizes to wreath products of the form $H \wr \Z$,
where $H$ is a finitely generated abelian group:
\begin{proposition}\Label{prop:NTFGAbWrZ}
  If $H$ is a finitely generated abelian group, the wreath product 
  $H \wr \Z$ is invariantly finitely $L$-presented.
\end{proposition}
\begin{proof}
  Since $H$ is finitely generated and abelian, it decomposes into a
  direct product of cyclic groups; i.e., $H$ has the form $\Z_{r_1}
  \times \cdots \times \Z_{r_n}$ for $r_1,\ldots,r_n\in\N \cup \{\infty\}$
  where $\Z_\infty$ denotes the infinite cyclic group while $\Z_{r_i}$
  denotes the cyclic group of order $r_i$, otherwise. Then $\la\X\mid \{
  [x,y] \mid x,y\in\X\} \cup \{ x^{r_x} \mid r_x < \infty \} \ra$ is a
  finite presentation for $H$. The wreath product $H\wr \Z$ 
  admits the presentation
  \[
    H \wr \Z \cong 
    \left\la \X \cup \{t\}~\middle|~ \{ [x,y], x^{r_x} \}_{x,y\in\X,r_x<\infty} \cup \{ [x,y^{t^i}]  \}_{
    x,y\in \X, i\in\N_0 } \right\ra.
  \]
  For each $y\in\X$, define a substitution $\sigma_y$ which is induced 
  by the map
  \[
    \sigma_y \colon \left\{\begin{array}{rcll}
      y &\mapsto& y^t\,y,\\
      x &\mapsto& x, &\textrm{for each } x \in \X \setminus \{y\},\\
      t &\mapsto& t.
    \end{array}\right.
  \]
  For $n \in \N$ and $x,y,z\in \X$ with $x\neq y$ and $z \neq y$, we obtain
  \begin{eqnarray*}
     [y,x^{t^{n}}]^{\sigma_y}         &=& [y^ty,x^{t^n}] = [y,x^{t^{n-1}}]^{ty}\cdot [y,x^{t^n}],\\ \relax
     [x,y^{t^{n}}]^{\sigma_y} &=& [x,y^{t^{n+1}}\,y^{t^{n}}] = [x,y^{t^{n}}]\cdot [x,y^{t^{n+1}}]^{y^{t^{n}}}, \\ \relax
     [x,z^{t^{n}}]^{\sigma_y} &=& [x,z^{t^{n}}], \\ \relax
     [y,y^{t^{n}}]^{\sigma_y} 
     &=& [ y,y^{t^{n-1}}]^{ty}\cdot [y,y^{t^{n}}]^{ty^{t^{n}}y} \cdot [y,y^{t^{n}} ]\cdot [y,y^{t^{n+1}}]^{y^{t^{n}}}.
  \end{eqnarray*}
  This shows that the relations $\{ [x,y^{t^i}] \mid x,y\in\X,
  i\in\N\}$ are consequences of the iterated images
  $\{[x,y^t]^{\delta}\mid \delta \in \{\sigma_y \mid y \in \X\}^*, x,y \in \X
  \}$ and vice versa. Moreover, for each relation $x^{r_x}$ of $H$'s
  finite presentation, we have that $(x^{r_x})^{\sigma_y} = x^{r_x}$
  if $x \neq y$ and $(y^{r_y})^{\sigma_y} = (y^ty)^{r_y} =_{H \wr \Z}
  (y^{r_y})^t\,y^{r_y}$, otherwise. Thus these images are relations of the
  wreath product $H \wr \Z$. In particular, the finite $L$-presentation
  \[
    \left\la \X \cup \{t\} ~\middle|~ \emptyset~\middle|~ \{ \sigma_y \}_{y\in \X}
    ~\middle|~ \{ [x,y^t] \}_{ x,y\in\X } \cup \{ x^{r_x} \}_{ x\in\X ,r_x<\infty} \right\ra
  \]
  is an invariant finite $L$-presentation for the wreath product $H \wr \Z$.
\end{proof}
Even though invariant finite $L$-presentations are known for
numerous self-similar groups, we are not aware of an invariant finite
$L$-presentation for the Gupta-Sidki group from~\cite{GS83}. Moreover, 
we are not aware of a finitely $L$-presented group which is not 
invariantly finitely $L$-presented. 

\section[Tietze Transformations for $L$-Presentations]{Tietze
         Transformations for \boldmath{$L$}-Presentations}
\Label{sec:NTietzeTr}
In this section, we introduce Tietze transformations for
$L$-presentations.  
Let $G = \la\X\mid\Q\mid\Phi\mid\R\ra$ be an $L$-presented
group. Denote by $F$ the free group $F(\X)$ over the alphabet $\X$ and let
$K = \la \Q \cup \bigcup_{\sigma\in\Phi^*} \R^\sigma\ra^F$ be the
kernel of the free presentation $\pi\colon F \to G$. Then $K
= \ker\pi$ decomposes into the normal subgroups $Q = \la \Q \ra^F$
and $R = \la \bigcup_{\sigma\in\Phi^*} \R^\sigma\ra^F$ so that $K
= RQ = QR$ holds. The group $F/R$ is invariantly $L$-presented by
$\la\X\mid\emptyset\mid\Phi\mid\R\ra$. We can add every element of
the kernel $K$ as a fixed relation:
\begin{proposition}\Label{prop:NTAddFxRel}
  If $G = \la\X\mid\Q\mid\Phi\mid\R\ra$ is a (finitely) $L$-presented
  group and\linebreak $\S \subseteq \la \Q \cup \bigcup_{\sigma\in\Phi^*}
  \R^\sigma\ra^F$ is a (finite) subset, then $\la \X \mid \Q \cup {\mc S}
  \mid\Phi\mid\R\ra$ is a (finite) $L$-presentation for $G$.
\end{proposition}
\begin{proof}
  The proof follows with the Tietze transformation that adds
  consequences $\S$ of $G$'s relations to the group presentation
  $\la\X\mid\Q\cup\bigcup_{\sigma\in\Phi^*}\R^\sigma\ra$.
\end{proof}
The transformation in Proposition~\ref{prop:NTAddFxRel} is reversible
in the sense that we can remove fixed relations ${\mc S}$ from an
$L$-presentation $\la\X\mid\Q\cup\S\mid\Phi\mid\R\ra$ if and only
if\linebreak $\la\Q\cup{\mc S}\cup\bigcup_{\sigma\in\Phi^*} \R^\sigma\ra^F
= \la \Q\cup\bigcup_{\sigma\in\Phi^*} \R^\sigma \ra^F$ holds. The
following transformations are reversible in the same sense.\smallskip

If an $L$-presentation is not invariant (cf. Remark~\ref{rem:NTNonInvLp}),
there exist elements from the kernel $K$ of the free presentation
$\pi\colon F \to G$  that we cannot add as iterated relations
without changing the isomorphism type of the group. However, even for
non-invariant $L$-presentations we have the following
\begin{proposition}
  \Label{prop:NTAddItRel}
  If $G = \la\X\mid\Q\mid\Phi\mid\R\ra$ is a (finitely) $L$-presented
  group and\linebreak $\S \subseteq \la \bigcup_{\sigma\in\Phi^*}
  \R^\sigma\ra^F$ is a (finite) subset, then $\la \X \mid \Q
  \mid\Phi\mid\R\cup\S\ra$ is a (finite) $L$-presentation for $G$.
\end{proposition}
\begin{proof}
  By construction, the normal subgroup $R = \la \bigcup_{\sigma\in\Phi^*}
  \R^\sigma\ra^F$ is $\sigma$-invariant for each $\sigma \in
  \Phi^*$. More precisely, for each $r \in R$ and $\sigma \in \Phi^*$, we
  have $r^\sigma \in R$. Therefore, adding the (possibly infinitely
  many) relations \mbox{$\{s^\sigma \mid s\in \S, \sigma\in\Phi^*\}$} to the
  group presentation $\la \X \mid \Q \cup \bigcup_{\sigma\in\Phi^*}
  \R^\sigma \ra$ does not change the isomorphism type of the group.
\end{proof}
Iterated and fixed relations of an $L$-presentation are related
by the following
\begin{proposition}\Label{prop:NTInterchFixIt}
  If $G = \la\X\mid\Q\mid\Phi\mid\R\ra$ is a (finitely) $L$-presented
  group and $\S \subseteq \R$ holds, then $\la\X\mid\Q\cup\S \mid\Phi\mid(\R
  \setminus\S)\cup\{r^\psi \mid r\in\S, \psi \in \Phi\}\ra$ is a (finite)
  $L$-presentation for $G$.
\end{proposition}
\begin{proof}
  The proof follows immediately from
  \[
    \Q \cup \bigcup_{\sigma\in\Phi^*} \R^\sigma
    = \Q \cup \S \cup \bigcup_{\sigma\in\Phi^*} \left((\R\setminus\S) \cup
      \{ r^\psi \}_{r\in \S, \psi\in\Phi} \right)^\sigma;
  \]
  these are the relations of $G$'s group presentation.
\end{proof}
The following proposition is a consequence of the definition of an
invariant $L$-presentation:
\begin{proposition}\Label{prop:NTReplRelsInvLp}
  If $\la\X\mid\Q\mid\Phi\mid\R\ra$ is an invariant
  (finite) $L$-presentation for the group $G$ and $\S \subseteq \Q$ holds,  then
  $\la\X\mid\Q\setminus\S\mid\Phi\mid\R\cup\S\ra$ is a (finite)
  $L$-presentation for $G$.
\end{proposition}
\begin{proof}
  Since $G$ is invariantly $L$-presented by
  $\la\X\mid\Q\mid\Phi\mid\R\ra$, each $\sigma \in \Phi$ induces an
  endomorphism of the group $G$. Therefore, the images \mbox{$\{q^\sigma
  \mid q \in \S, \sigma \in \Phi^*\}$} are relations within $G$ and
  so $\la\X \mid (\Q\setminus\S) \cup \bigcup_{\sigma\in\Phi^*}
  (\R\cup\S)^\sigma\ra$ is a presentation for $G$.
\end{proof}
The following proposition allows us to add generators together with
fixed relations to an $L$-presentation:
\begin{proposition}\Label{prop:NTAddGensFx}
  Let $G = \la\X\mid\Q\mid\Phi\mid\R\ra$ be an $L$-presented group, ${\mc
  Z}$ be an alphabet so that $\X \cap {\mc Z} = \emptyset$ holds, and, for
  each $z \in {\mc Z}$, let $w_z \in F(\X)$ be given. For each 
  $\sigma \in \Phi$, define an endomorphism of the free group $E$
  over the alphabet $\X \cup{\mc Z}$ that is induced by the map
  \begin{equation}
    \ti\sigma\colon \left\{ \begin{array}{rcll}
      x &\mapsto& x^\sigma, &\textrm{ for each } x \in \X,\\
      z &\mapsto& g_z,&\textrm{ for each } z \in {\mc Z}, 
    \end{array}\right.
    \Label{eqn:NTAddGensFx}
  \end{equation}
  where $g_z$ are arbitrary elements of the free group $E$.
  Then $G$ satisfies that
  \begin{equation}
    G \cong \la\,\X\cup{\mc Z}\mid \Q\cup\{z^{-1} w_z\}_{ z \in {\mc Z}
    } \mid \{\ti\sigma \}_{\sigma\in\Phi}\mid\R\,\ra.
    \Label{eqn:NTAddGensFxLp}
  \end{equation}
  If $\la\X\mid\Q\mid\Phi\mid\R\ra$ is a finite $L$-presentation 
  and ${\mc Z}$ is a finite alphabet, the $L$-presentation in
  Eq.~(\ref{eqn:NTAddGensFxLp}) is finite.
\end{proposition}
\begin{proof}
  Write $H = \la \X\cup{\mc Z}\mid \Q\cup\{z^{-1} w_z\mid z\in{\mc
  Z}\} \mid\{\ti\sigma\mid \sigma\in\Phi\}\mid\R\ra$ and let $F$ and
  $E$ be the free groups over $\X$ and $\X\cup{\mc Z}$,
  respectively. To avoid confusion, the elements of $G$'s presentation
  are denoted by $\overline{g}\in F$. Then 
  \[
    \pi\colon \left\{\begin{array}{rcll}
      x &\mapsto& \overline{x}, &\textrm{for each } x \in \X, \\
      z &\mapsto& \overline{w_z}, &\textrm{for each }z \in {\mc Z},
    \end{array}\right.
  \]
  induces a surjective homomorphism $\pi\colon E \to F$.  By construction,
  the restriction of the substitution $\ti\sigma$ to the free group $F$ coincides with $\sigma$.
  Thus $\left(\bigcup_{\sigma\in\Phi} \R^{\ti\sigma}\right)^\pi
  = \bigcup_{\sigma\in\Phi^*} \R^\sigma$ and hence, $\pi$ maps
  iterated relations of $H$'s $L$-presentation to iterated relations
  of $G$. Similarly, $\pi$ maps the fixed relations $\Q$ of $H$'s
  $L$-presentation to fixed relations of $G$. It remains to consider
  the relations of the form $z^{-1}w_z$ with $z\in{\mc Z}$. However,
  these relations are mapped trivially by $\pi$. This shows that the
  homomorphism $\pi\colon E \to F$ induces a surjective homomorphism
  $\ti\pi\colon H \to G$. On the other hand, identifying the generators of
  $G$'s $L$-presentation with the generators of $H$ induces a surjective
  homomorphism $\varphi\colon G\to H$ with $\varphi\ti\pi = \id_H$
  and $\ti\pi\varphi = \id_G$. Hence, the groups $G$ and $H$
  are isomorphic. The second assertion is obvious.
\end{proof}
We can also add the relations $\{z^{-1} w_z \mid z \in {\mc Z}\}$
in Proposition~\ref{prop:NTAddGensFx} as iterated relations to the
$L$-presentation if we define the substitutions $\ti\sigma$ as follows:
\begin{proposition}\Label{prop:NTAddGensIt}
  Let $G = \la\X\mid\Q\mid\Phi\mid\R\ra$ be an $L$-presented group,
  ${\mc Z}$ be an alphabet so that $\X \cap {\mc Z} = \emptyset$ holds, and, for
  each $z \in {\mc Z}$, let $w_z \in F(\X)$ be given. For each 
  $\sigma \in \Phi$, define an endomorphism of the free group $E$
  over the alphabet $\X \cup {\mc Z}$ that is induced by the map
  \begin{equation}
    \ti\sigma\colon \left\{ \begin{array}{rcll}
      x &\mapsto& x^\sigma, &\textrm{ for each } x \in \X,\\
      z &\mapsto& w_z^\sigma, &\textrm{ for each } z \in {\mc Z}. 
    \end{array}\right.\Label{eqn:NTAddGensIt}
  \end{equation}
  Then $G$ satisfies that 
  \begin{equation}
    G \cong \la\,\X\cup{\mc Z}\mid \Q \mid\{\ti\sigma \}_{
    \sigma\in\Phi}\mid\R\cup\{z^{-1} w_z\}_{ z \in {\mc Z} }\,\ra.
    \Label{eqn:NTAddGensItLp}
  \end{equation}
  If $\la\X\mid\Q\mid\Phi\mid\R\ra$ is a finite $L$-presentation 
  and ${\mc Z}$ is a finite alphabet, the $L$-presentation in
  Eq.~(\ref{eqn:NTAddGensItLp}) is finite.
\end{proposition}
\begin{proof}
  The substitutions $\ti\sigma$ in Eq.~(\ref{eqn:NTAddGensIt}) are
  well-defined because $w_z \in F(\X)$ and \mbox{$\sigma \in \End(F(\X))$}
  hold. By Proposition~\ref{prop:NTInterchFixIt}, we have that
  \begin{eqnarray*}
    && \left\la \X\cup{\mc Z}~\middle|~\Q~\middle|~\{\ti\sigma\}_{\sigma\in\Phi}
      ~\middle|~\R\cup\{z^{-1} w_z \}_{ z \in {\mc Z}}\right\ra \\[0.5ex]
    &=& \left\la \X\cup{\mc Z}~\middle|~ \Q\cup\{z^{-1} w_z\}_{ z \in {\mc Z}}
      ~\middle|~\{\ti\sigma\}_{\sigma\in\Phi}~\middle|~\R\cup
      \{(z^{-1} w_z)^{\ti\sigma} \}_{z\in {\mc Z}, \sigma\in\Phi}\right\ra.
  \end{eqnarray*}
  By definition of $\ti\sigma$ in Eq.~(\ref{eqn:NTAddGensIt}), we have
  $(z^{-1})^{\ti\sigma} = (w_z^\sigma)^{-1}$ and $w_z^{\ti\sigma} =
  w_z^\sigma$. Thus $(z^{-1}\,w_z)^{\ti\sigma} = (w_z^\sigma)^{-1} w_z^\sigma =
  1$ holds. In particular, adding the relations
  $\{(z^{-1}w_z)^{\ti\sigma}\mid z\in{\mc Z},\sigma\in\Phi\}$ to a group
  presentation does not change the isomorphism type of the group. By
  Proposition~\ref{prop:NTAddGensFx}, we have that
  \begin{eqnarray*}
    G&=&\left\la \X ~\middle|~ \Q ~\middle|~\Phi~\middle|~\R\right\ra\\ 
    &\cong&\left\la \X\cup{\mc Z}~\middle|~ \Q\cup\{z^{-1} w_z \}_{z\in{\mc Z}} ~\middle|~\{\ti\sigma\}_{\sigma\in\Phi}~\middle|~\R\right\ra \\
    &=&\left\la \X\cup{\mc Z}~\middle|~ \Q\cup\{z^{-1} w_z \}_{ z\in {\mc Z}} ~\middle|~\{\ti\sigma\}_{\sigma\in\Phi}~\middle|~\R\cup\{ (z^{-1} w_z)^{\ti\sigma} \}_{ z\in {\mc Z},\sigma\in\Phi }\right\ra \\
    &=& \left\la \X\cup{\mc Z}~\middle|~ \Q ~\middle|~\{\ti\sigma\}_{\sigma\in\Phi}~\middle|~\R\cup\{z^{-1} w_z\}_{z \in {\mc Z}}\right\ra;
  \end{eqnarray*}
  which proves the first assertion of Proposition~\ref{prop:NTAddGensIt}
  while the second is obvious.
\end{proof}
The following proposition allows
us to modify the substitutions of an $L$-presentation:
\begin{proposition}
  \Label{prop:NTModPhi}
  If $G = \la\X\mid\Q\mid\Phi\mid\R\ra$ is a (finitely) $L$-presented group
  and $\Psi \subseteq \Phi$ holds, then 
  $\la\X\mid\Q\mid(\Phi\setminus\Psi)\cup\{\sigma\psi\mid \psi\in\Psi, \sigma\in\Phi\} \mid \R \cup \bigcup_{\psi \in \Psi} \R^\psi\ra$
  is a (finite) $L$-presentation for $G$. 
\end{proposition}
\begin{proof}
  The proof follows immediately from 
  \[
     \Q \cup \bigcup_{\sigma\in\Phi^*} \R^\sigma 
   = \Q \cup \bigcup_{\sigma\in\widehat\Phi^*}\Big( \R \cup \bigcup_{\psi \in \Psi} \R^\psi \Big)^\sigma
  \]
  where $\widehat\Phi = (\Phi\setminus\Psi)\cup\{\sigma\psi\mid \psi\in\Psi,
  \sigma\in\Phi\}$; these are the relations of $G$'s group presentation.
\end{proof}
Since each relation of a group presentation can be replaced by a
conjugate, we can modify the substitutions of an $L$-presentation 
as follows:
\begin{proposition}\Label{prop:NTModPhiConj}
  Let $G = \la\X\mid\Q\mid\Phi\mid\R\ra$ be a (finitely) $L$-presented
  group, $\S \subseteq F$ be a (finite) subset, and let $\Psi \subseteq
  \Phi$ be given. For each $x \in \S$, denote by $\delta_x$ the inner
  automorphism of the free group $F(\X)$ that is induced by conjugation with
  $x$. Then
  \begin{itemize}\addtolength{\itemsep}{-1ex}
  \item $\la\X\mid\Q\mid\Phi\cup\{\delta_x \mid x \in \S\} \mid \R\ra$,
  \item $\la\X\mid\Q\mid(\Phi\setminus\Psi)\cup \{\delta_x\sigma \mid x \in \S, \sigma \in \Psi\}\mid \R\ra$,
  and
  \item $\la\X\mid\Q\mid(\Phi\setminus\Psi)\cup \{\sigma\delta_x \mid x \in \S, \sigma \in \Psi\}\mid \R\ra$
  \end{itemize}
  are (finite) $L$-presentations for $G$.
\end{proposition}
\begin{proof}
  This follows because each relation of a group presentation can be
  replaced by a conjugate and we have $\delta_x \sigma = \sigma
  \delta_{x^\sigma}$ for each $\sigma \in \Phi^*$ and $x\in\X$.
\end{proof}
Recall that the kernel $K =
\la\Q\cup\bigcup_{\sigma\in\Phi^*} \R^\sigma\ra^F$ of the free
presentation $\pi\colon F \to G$ decomposes into the normal subgroups $Q =
\la\Q\ra^F$ and $R = \la \bigcup_{\sigma\in\Phi^*} \R^\sigma\ra^F$
so that $K = QR = RQ$ holds. This decomposition yields the following
\begin{proposition} \Label{prop:NTAddSubst}
  Let $G = \la\X\mid\Q\mid\Phi\mid\R\ra$ be a (finitely) $L$-presented
  group and let $\Psi \subseteq \End(F(\X))$ be a (finite) subset. If
  each \mbox{$\psi \in \Psi$} induces an endomorphism of $F(\X)/R$, then
  $\la\X\mid\Q\mid\Phi\cup\Psi\mid\R\ra$ is a (finite) $L$-presentation
  for $G$.
\end{proposition}
\begin{proof}
  If $\psi \in \Psi$ induces an endomorphism of $F(\X)/R$, the normal
  subgroup $R$ is $\psi$-invariant. Therefore, each image $r^{\sigma} \in
  F(\X)$, with $\sigma \in (\Phi\cup\Psi)^*\setminus \Phi^*$ and $r\in
  \R$, is a relation of the group. Adding these (possibly infinitely many)
  relations to the group presentation does not change the isomorphism
  type of the group.
\end{proof}
For an invariant $L$-presentation, we even have the following
\begin{proposition}\Label{prop:NTAddSubstInv}
  Let $G = \la\X\mid\Q\mid\Phi\mid\R\ra$ be a (finitely)
  $L$-presented group and let\linebreak $\Psi \subseteq \End(F(\X))$ be a (finite)
  subset. Then \mbox{$\la\X\mid\Q\mid\Phi\cup\Psi\mid\R\ra$} is a
  (finite) $L$-pre\-sen\-ta\-tion for $G$ if and only if each $\psi\in\Psi$
  induces an endomorphism of $G$.
\end{proposition}
\begin{proof}
  Let $K = \la \Q \cup \bigcup_{\sigma\in\Phi^*} \R^\sigma \ra^F$
  be the kernel of the free presentation $\pi\colon F(\X)
  \to G$.  If each $\psi \in \Psi$ induces an endomorphism
  of $F(\X)/K$, Proposition~\ref{prop:NTAddSubst} shows
  the first assertion. If, on the other hand, the invariant
  $L$-presentations \mbox{$\la\X\mid\Q\mid\Phi\mid\R\ra$} and
  \mbox{$\la\X\mid\Q\mid\Phi\cup\Psi\mid\R\ra$} are $L$-presentations
  for $G$, each $\psi \in \Psi$ induces an endomorphism of $G = F(\X)/K$.
\end{proof}
Every substitution $\sigma \in \Phi$ of an invariant 
$L$-presentation $G = \la\X\mid\Q\mid\Phi\mid\R\ra$ induces an
endomorphism of $G$. However, there are possibly other endomorphisms
of the free group $F(\X)$ that will induce the same endomorphism on
$G$. The following proposition allows us to modify a given substitution
of an $L$-presentation:
\begin{proposition}\Label{prop:NTModSubst}
  Let $G = \la\X\mid\Q\mid\Phi\mid\R\ra$ be a (finitely) $L$-presented
  group,\linebreak ${\mc S} \subseteq \la \bigcup_{\varphi\in\Phi^*} \R^\varphi
  \ra^F$ be a (finite) subset, and let $\sigma \in \Phi$ be given.
  Define an endomorphism $\ti\sigma$ of the free group $F=F(\X)$ over
  the alphabet $\X$ that is induced by the map $\ti\sigma\colon x
  \mapsto x^\sigma\,r_x$ for each $x \in \X$ and some $r_x \in \S$. Then
  $\la\X\mid\Q\mid(\Phi\setminus\{\sigma\}) \cup\{\ti\sigma\}\mid \R
  \cup \S \ra$ is a (finite) $L$-presentation for $G$.
\end{proposition}
\begin{proof}
  We work in the free group $F = F(\X)$ over the alphabet $\X$ and we
  decompose the kernel $K = \la \Q \cup \bigcup_{\varphi\in\Phi^*}
  \R^\varphi \ra^F$ of the free presentation $\pi\colon F \to G$ into the
  normal subgroups $Q = \la\Q\ra^F$ and $R = \la \bigcup_{\varphi\in\Phi^*}
  \R^\varphi \ra^F$ as above. Since $\S \subseteq \la\bigcup_{\varphi\in\Phi^*}
  \R^\varphi \ra^F$ holds, Proposition~\ref{prop:NTAddItRel} yields that
  $G = \la\X\mid\Q\mid\Phi\mid\R\ra = \la\X\mid\Q\mid\Phi\mid\R\cup\S\ra$.
  In particular, we have that $R = \la \bigcup_{\varphi\in\Phi^*}
  (\R \cup \S)^\varphi \ra^F$. Write $\Psi = (\Phi \setminus \{\sigma\}
  ) \cup \{\ti\sigma\}$. We prove this proposition by showing that the
  normal subgroups $R = \la\bigcup_{\varphi\in\Phi^*} (\R \cup \S)^\varphi
  \ra^F$ and $\ti R = \la\bigcup_{\varphi\in\Psi^*} (\R \cup \S)^\varphi
  \ra^F$ coincide. For this purpose, we prove that, for each $\ti\delta
  \in \Psi^*$ and $g \in F$, there exists $\delta \in \Phi^*$ and $h
  \in L = \la \bigcup_{\varphi \in \Phi^*} \S^\varphi \ra^F$ so that
  $g^{\ti\delta} = g^\delta h$ holds. By construction, we have that $L
  \subseteq R$. By symmetry (as we have both $x^{\ti\sigma} = x^\sigma
  r_x$ and $x^{\sigma} = x^{\ti\sigma} r_x^{-1}$) the same arguments
  will show that, for each $\delta\in\Phi^*$ and $g \in F$, there
  exists $\ti\delta \in \Psi^*$ and $h \in \ti L = \la \bigcup_{\varphi
  \in\Psi^*} \S^{\varphi } \ra^F$ so that $g^\delta = g^{\ti\delta} h$
  holds. This would yield that each normal generator $s^{\ti\delta}\in
  \ti R$, with $s \in \R \cup \S$ and $\ti\delta \in \Psi^*$, can be
  written as $s^{\ti\delta} = s^\delta h$ for some $\delta \in \Phi^*$
  and $h \in L \subseteq R$. In fact, $s^{\ti\delta}\in\ti R$ satisfies that
  $s^{\ti\delta} = s^\delta h \in R$ and thus $\ti R \subseteq R$. By
  symmetry, we would also obtain that $R \subseteq \ti R$ holds. This
  clearly proves Proposition~\ref{prop:NTModSubst}.\smallskip

  It therefore remains to prove that, for each $\ti\delta \in \Psi^*$
  and $g \in F$, there exists $\delta\in\Phi^*$ and $h\in L$ so that
  $g^{\ti\delta} = g^\delta h$ holds. Each $g \in F$ is represented
  by a finite word $w_g(x_{i_1},\ldots,x_{i_n})$ over finitely many
  generators $\{x_{i_1},\ldots,x_{i_n}\} \subseteq \X$.  Let $\ti\delta
  \in \Psi^*$ and $g \in F$ be given. We prove the assertion by induction
  on $m = \|\ti\delta\|$. If $m = 1$, then $\ti\delta \in \Psi$. Moreover, we
  either have $\ti\delta = \ti\sigma$ or $\ti\delta \neq \ti\sigma$.
  If $\ti\delta \neq \ti\sigma$ holds, then $\ti\delta \in \Phi$ and thus
  $g^{\ti\delta} = g^{\delta} h$ for some $\delta \in \Phi$ and
  $h \in L$. Otherwise, if $\ti\delta = \ti\sigma$ holds, we obtain that
  \[
    g^{\ti\sigma} = w_g(x_{i_1},\ldots,x_{i_n})^{\ti\sigma} 
    = w_g(x_{i_1}^{\ti\sigma},\ldots,x_{i_n}^{\ti\sigma}) 
    = w_g(x_{i_1}^{\sigma}\,r_{x_{i_1}},\ldots,x_{i_n}^{\sigma}\,r_{x_{i_n}}).
  \]
  Conjugation in the free group $F$ yields that the
  word $w_g(x_{i_1}^{\sigma}\,r_{x_{i_1}}, \ldots,
  x_{i_n}^{\sigma}\,r_{x_{i_n}})$ can be written as
  $w_g(x_{i_1}^{\sigma},\ldots,x_{i_n}^{\sigma}) \cdot h$ for some $h
  \in \la \S \ra^F$. Thus $g^{\ti\sigma} = g^\sigma \cdot h$ for some
  $\sigma \in \Phi$ and $h \in \la \S \ra^F \subseteq L$.\smallskip

  For an integer $m > 1$, assume that, for every $g\in F$ and $\ti\delta
  \in \Psi^*$, with $\|\ti\delta\| = m$, the image $g^{\ti\delta}\in\ti R$
  satisfies that $g^{\ti\delta} = g^\delta h$ for $\delta \in
  \Phi^*$ and some $h \in L$. Let $g \in F$ and $\ti\delta \in \Psi^*$,
  with $\|\ti\delta\| = m+1$, be given. Then there exist $\ti\omega
  \in \Psi$ and $\ti\gamma \in \Psi^*$, with $\|\ti\gamma\| = n$, so
  that $\ti\delta = \ti\gamma\,\ti\omega$ holds. By our assumption,
  there exist $\gamma \in \Phi^*$ and $h \in L$ so that $g^{\ti\gamma}
  = g^\gamma h$ holds. Thus $g^{\ti\delta} = g^{\ti\gamma\,\ti\omega}
  = (g^\gamma h)^{\ti\omega}$.  If $\ti\omega \neq \ti\sigma$ holds,
  then $\ti\omega \in \Phi$ and thus $\gamma \ti \omega \in
  \Phi^*$. Moreover, by construction, the normal subgroups $L =
  \la \bigcup_{\varphi \in \Phi^*} \S^\varphi \ra^F$ and $\ti L =
  \la \bigcup_{\varphi \in \Psi^*} \S^\varphi \ra^F$ are $\Phi^*$-
  and $\Psi^*$-invariant, respectively. Thus $h^{\ti\omega} \in L$
  if $\ti\omega \neq \ti\sigma$. Therefore, the image $g^{\ti\delta}$
  satisfies that $g^{\ti\delta} = g^{\gamma\ti\omega} h^{\ti\omega}$
  for some $\gamma\ti\omega\in\Phi^*$ and $h^{\ti\omega} \in L$.
  It suffices to consider the case $\ti\omega = \ti\sigma$.
  The elements $g^{\gamma} \in F$ and $h\in F$ are represented
  by finite words $w_{g^\gamma}(x_{j_1},\ldots,x_{j_n})$ and
  $w_h(x_{k_1},\ldots,x_{k_\ell})$, respectively. Again,
  conjugation in the free group $F$ yields that
  $w_{g^\gamma}(x_{j_1},\ldots,x_{j_n})^{\ti\sigma} =
  w_{g^\gamma}(x_{j_1}^\sigma,\ldots,x_{j_n}^\sigma)\,
  u$ and $w_h(x_{k_1},\ldots,x_{k_\ell})^{\ti\sigma} =
  w_h(x_{k_1}^\sigma,\ldots,x_{k_\ell}^\sigma)\,v$ with $u,v \in \la {\mc
  S}\ra^F$. Thus $g^{\ti\delta} = g^{\ti\gamma\ti\sigma} = (g^\gamma
  h)^{\ti\sigma} = (g^{\gamma\sigma} u)\,(h^\sigma\,v)$. In fact, we
  have that $g^{\ti\delta} = g^{\gamma\sigma} \, h'$ with $\gamma\sigma
  \in \Phi^*$ and $h' = uh^\sigma v\in L$. Thus, for every $g \in F$
  and $\ti\delta \in \Psi^*$, the image $g^{\ti\delta}$ satisfies that
  $g^{\ti\delta} = g^\delta h$ with $\delta \in \Phi^*$ and $h\in L$. By
  symmetry, as we have both $x^{\ti\sigma} = x^\sigma\,r_x$ and $x^\sigma
  = x^{\ti\sigma}\,r_x^{-1}$, the same arguments will prove that for each
  $g \in F$ and $\delta \in \Phi^*$ the image $g^\delta$ satisfies
  that $g^\delta = g^{\ti\delta} h$ with $\ti\delta \in \Psi^*$ and $h
  \in \ti L = \la \bigcup_{\varphi\in\Psi^*} \S^\varphi \ra^F$. This
  finishes our proof of Proposition~\ref{prop:NTModSubst}.
\end{proof}
As a consequence of Proposition~\ref{prop:NTModSubst}, we obtain 
the following 
\begin{corollary}
  Let $G = \la \X\mid\Q\mid\Phi\mid\R\ra$ be a finitely $L$-presented group
  and let\linebreak $\sigma \in \Phi$ be given. Then $\sigma$ induces
  an endomorphism of the invariantly finitely $L$-presented group $H =
  \la\X\mid\emptyset\mid\Phi\mid\R\ra$. If $\psi\in \End(F(\X))$
  and $\sigma$ induce the same endomorphism on $H$, then there exists a
  finite subset $\S \subseteq F(\X)$ so that $\la \X\mid\Q\mid
  (\Phi\setminus\{\sigma\})\cup\{\psi\}\mid\R \cup \S\ra$ is a 
  finite $L$-presentation for $G$.
\end{corollary}
\begin{proof}
  If $\sigma$ and $\psi$ induce the same endomorphism of $H$,
  there exists, for each $x \in \X$, an element $r_x \in \la
  \bigcup_{\sigma\in\Phi^*} \R^\sigma \ra^F$ with $x^\psi
  = x^\sigma r_x$.  Write $\S = \{ r_x \mid x \in \X\}$. Then
  Proposition~\ref{prop:NTModSubst} yields that $G = \la \X \mid \Q \mid
  (\Phi\setminus\{\sigma\})\cup\{\psi\}\mid \R\cup \S\ra$.
\end{proof}
The transformations introduced above allow us to modify a given
$L$-presentation of a group. In order to prove Tietze's theorem for
invariantly finitely $L$-presented groups, we choose the following set
of transformations:
\begin{definition}\Label{def:NTietzeTr}
  An \emph{$L$-Tietze transformation} is a transformation that 
  \begin{enumerate}\addtolength{\itemsep}{-1ex}
  \item adds or removes a single fixed relation (Proposition~\ref{prop:NTAddFxRel}),
  \item adds or removes a single iterated relation (Proposition~\ref{prop:NTAddItRel}), 
  \item adds or removes a single substitution (Proposition~\ref{prop:NTAddSubst}),
  \item adds or removes a generator together with a fixed relation (Proposition~\ref{prop:NTAddGensFx}), 
  \item adds or removes a generator together with an iterated relation (Proposition~\ref{prop:NTAddGensIt}), or that
  \item modifies a given substitution of an $L$-presentation (Proposition~\ref{prop:NTModSubst}).
  \end{enumerate}
\end{definition}

\section{Applications of Tietze Transformations}\Label{sec:NTAppsTie}
The transformations introduced in Section~\ref{sec:NTietzeTr} allow
us to prove Theorem~\ref{thm:LTietzeInv}:
\def\0{ of Theorem~\ref{thm:LTietzeInv}}
\begin{CenProof}{\0}
  We use similar ideas as in the proof of Tietze's theorem
  in~\cite[Chapter~\Rom{2}]{LS77}: As each $L$-Tietze transformation
  does not change the isomorphism type of the group, two finite
  $L$-presentations define isomorphic groups if one $L$-presentation can be
  transformed into the other by a finite sequence of $L$-Tietze
  transformations. In order to prove Theorem~\ref{thm:LTietzeInv},
  it suffices to prove that two invariant finite
  $L$-presentations which define isomorphic groups can be
  transformed into each other by a finite sequence of $L$-Tietze
  transformations.  For this purpose, we consider two invariant finite
  $L$-presentations \mbox{$\la\X_1\mid\Q_1\mid\Phi_1\mid\R_1\ra$}
  and \mbox{$\la\X_2\mid\Q_2\mid\Phi_2\mid\R_2\ra$} of a group $G$. By
  Proposition~\ref{prop:NTReplRelsInvLp}, we can assume that both $\Q_1 =
  \emptyset$ and $\Q_2 = \emptyset$ hold. We will construct an invariant
  finite $L$-presentation for $G$ which can be obtained from both $L$-presentations 
  by a finite sequence of $L$-Tietze transformations. Because each
  $L$-Tietze transformation is reversible, this shows that we can
  pass from one $L$-presentation to the other by a finite sequence of
  $L$-Tietze transformations.\smallskip

  Suppose that $\X_1 \cap \X_2 = \emptyset$ holds. For $i\in\{1,2\}$,
  we denote by $F_i = F(\X_i)$ the free group over the alphabet $\X_i$ and 
  by $\pi_i\colon F_i \to G$ the free presentation with kernel
  $\ker(\pi_i) = \la \bigcup_{\sigma\in \Phi_i^*} \R_i^\sigma \ra^{F_i}$.
  For each $x \in\X_1$, we choose $w_x \in F_2$ with $x^{\pi_1} =
  w_x^{\pi_2}$; i.e., the element $w_x \in F_2$ is a $\pi_2$-preimage
  of $x^{\pi_1} \in G$. For each $z\in \X_2$, we choose $w_z
  \in F_1$ with $z^{\pi_2} = w_z^{\pi_1}$. Define the subsets
  $S_1 = \{x^{-1} w_x \mid x\in\X_1\}$ and
  $S_2 = \{z^{-1} w_z \mid z\in\X_2\}$
  of the free group $F = F(\X_1\cup\X_2)$ over the alphabet $\X_1\cup \X_2$.
  By Proposition~\ref{prop:NTAddGensIt}, we can add the finitely many
  generators $z \in \X_2$ together with the iterated relation $z^{-1}\,w_z
  \in S_2$ if we extend each substitution $\sigma \in \Phi_1$ to the
  free group $F$ by
  \[
    \ti\sigma\colon \left\{\begin{array}{rcll}
      x &\mapsto& x^\sigma, &\textrm{for each }x \in \X_1,\\
      z &\mapsto& w_z^\sigma, &\textrm{for each } z \in \X_2. 
    \end{array}\right.
  \]
  This yields the finite $L$-presentation 
  \[
    \la\,\X_1 \cup \X_2 \mid \emptyset \mid \{ \ti\sigma \}_{\sigma \in \Phi_1} 
    \mid \R_1 \cup \{ z^{-1} w_z \}_{z \in \X_2}\,\ra 
  \]
  for the group $G$. The natural homomorphisms $\pi_1\colon F_1 \to
  G$ and $\pi_2\colon F_2 \to G$ extend to a natural homomorphism
  $\pi\colon F \to G$ that is induced by the map
  \[ 
    \pi\colon\left\{\begin{array}{rcll}
      x &\mapsto& x^{\pi_1}, &\textrm{for each }x \in \X_1,\\
      z &\mapsto& z^{\pi_2}, &\textrm{for each }z \in \X_2.
    \end{array}\right.
  \]   
  Its kernel satisfies $\ker(\pi) =
  \la\bigcup_{\sigma\in\Phi_1^*} (\R_1\cup S_2)^{\ti\sigma}\ra^F$.
  For $x\in\X_1$ and $x^{-1} w_x \in S_1$, we have $x^\pi =
  x^{\pi_1} = w_x^{\pi_2} = w_x^\pi$ and thus $x^{-1} w_x \in \ker(\pi)$ holds.
  For each $r \in \R_2$, we have $r^{\pi} = r^{\pi_2} = 1$ and thus $r\in
  \ker(\pi)$ holds. Since the kernel $\ker(\pi)$ is $\{\ti\sigma\mid \sigma \in
  \Phi_1\}^*$-invariant, by construction, Proposition~\ref{prop:NTAddItRel}
  yields that
  \[
     G \cong \left\la\,\X_1 \cup \X_2 ~\middle|~ \emptyset ~\middle|~ \{\ti\sigma\}_{
     \sigma\in\Phi_1}~\middle|~ \R_1\cup \R_2 \cup S_1 \cup S_2
     \,\right\ra.
  \]
  As the invariant finite $L$-presentations
  $\la\X_1\mid\emptyset\mid\Phi_1\mid\R_1\ra$ and
  $\la\X_2\mid\emptyset\mid\Phi_2\mid\R_2\ra$ define isomorphic groups
  and every $\psi \in \Phi_2$ induces an endomorphism of the whole group,
  we can extend $\psi$ to an endomorphism of the free group $F$ over the
  alphabet $\X_1 \cup \X_2$ that induces the same endomorphism on $G$
  as $\psi$ does. More precisely, for each $\psi \in \Phi_2$, we define
  an endomorphism of the free group $F$ that is induced by the map
  \[
   \ti\psi\colon \left\{\begin{array}{rcll}
     z &\mapsto& z^\psi,   &\textrm{for each }z \in \X_2\\
     x &\mapsto& w_x^\psi, &\textrm{for each }x \in \X_1\textrm{ and }x^{-1}w_x \in S_1.
   \end{array}\right.
  \]
  By construction, the normal subgroup $\la\bigcup_{\sigma\in
  \Phi_1^*} (\R_1\cup\R_2\cup S_1\cup S_2)^{\ti\sigma} \ra^F$ is
  $\ti\psi$-invariant.  Thus, by Proposition~\ref{prop:NTAddSubst}, the group
  $G$ satisfies that
  \begin{equation}
    G \cong \left\la\X_1 \cup \X_2 ~\middle|~ \emptyset ~\middle|~ \{\ti\sigma\}_{\sigma\in\Phi_1}
    \cup\{\ti\psi\}_{\psi\in\Phi_2}~\middle|~ \R_1\cup \R_2 \cup S_1 \cup
    S_2 \right\ra. \Label{eqn:NTTietze}
  \end{equation}
  Since the $L$-presentations $\la\X_1\mid\Q_1\mid\Phi_1\mid\R_1\ra$
  and $\la\X_2\mid\Q_2\mid\Phi_2\mid\R_2\ra$ were finite, we
  have applied only finitely many $L$-Tietze transformations
  from Definition~\ref{def:NTietzeTr}. Therefore, starting 
  with the $L$-presentation $\la\X_1\mid\Q_1\mid\Phi_1\mid\R_1\ra$ 
  we have obtained the $L$-presentation in Eq.~(\ref{eqn:NTTietze})
  after finitely many steps. 
  By symmetry, though, we would
  also obtain the finite $L$-presentation in Eq.~(\ref{eqn:NTTietze})
  if we would have started with the finite $L$-presentation
  \mbox{$\la\X_2\mid\Q_2\mid\Phi_2\mid\R_2\ra$}. Since each
  $L$-Tietze transformation is reversible, we can therefore transform the finite
  $L$-pre\-sen\-ta\-tion in Eq.(\ref{eqn:NTTietze}) to the finite
  $L$-presentation $\la\X_2\mid\Q_2\mid\Phi_2\mid\R_2\ra$. This
  yields a finite sequence of $L$-Tietze transformations
  that allows us to transform the $L$-presentation
  $\la\X_1\mid\Q_1\mid\Phi_1\mid\R_1\ra$ to the $L$-presentation
  $\la\X_2\mid\Q_2\mid\Phi_2\mid\R_2\ra$ and vice versa.
\end{CenProof}
\noindent Similarly, the Tietze transformations in Section~\ref{sec:NTietzeTr}
also allow us to prove that two arbitrary finite $L$-presentations
could be transformed into each other by a finite sequence of Tietze
transformations.\smallskip

Another application of $L$-Tietze transformations is to prove that
`being invariantly finitely $L$-presented' is an abstract property
of a group that does not depend on the generating set of the group;
that is, if a group admits an invariant finite $L$-presentation with
respect to one finite generating set, then so it does
with respect to any other finite generating set. This result
was already posed in~\cite[Proposition~2.2]{Bar03}. However,
its proof contains a gap: Consider the invariant finite $L$-presentation
\[
  \Grig \cong \la \{a,b,c,d\} \mid \{a^2,b^2,c^2,d^2,bcd\} \mid
  \{ \sigma \} \mid \{ (ad)^4, (adacac)^4 \}  \ra 
\]
from Theorem~\ref{thm:Lyseniok}, where $\sigma$ is induced by the map $a\mapsto aca$,
$b\mapsto d$, $c\mapsto b$, and $d\mapsto c$. Then $\sigma$ is a
monomorphism of the free group $F = F(\{a,b,c,d\})$. The transformations
in the proof of~\cite[Proposition~2.2]{Bar03} keep the rank of
$\im(\sigma)$ constant and therefore, they do not allow to prove that
the Grigorchuk group admits an invariant finite $L$-presentation with
generators $\{a,c,d\}$ as in Eq.~(\ref{eqn:GrigACD}). The $L$-Tietze
transformations from Section~\ref{sec:NTietzeTr} allow us to address this gap:
\def\0{Theorem~\ref{thm:AbstProp}}
\begin{CenProof}{ of \0}
  Let ${\mc Y} = \{y_1,\ldots,y_n\}$ be an arbitrary finite generating
  set of the invariantly finitely $L$-presented group $G = \la \X \mid
  \Q \mid \Phi \mid \R\ra$. As $G$ is invariantly $L$-presented, we can
  assume that $\Q = \emptyset$ holds. Since ${\mc Y}$ generates $G$,
  there exists, for each $x \in \X$,  a word $w_x(y_1,\ldots,y_n)$ over the
  generators ${\mc Y}$ so that $x =_G w_x(y_1,\ldots,y_n)$ holds. 
  Since $\X = \{x_1,\ldots,x_m\}$ also generates $G$, there exists, for
  each $y \in {\mc Y}$, a word $w_y(x_1,\ldots,x_m)$ so
  that $y =_G w_y(x_1,\ldots,x_m)$ holds. Suppose that $\X \cap {\mc Y}
  = \emptyset$ holds.  For each $\sigma \in \Phi$, define
  an endomorphism $\ti\sigma$ of the free group $E$ over the alphabet
  $\X \cup {\mc Y}$ that is induced by the map
  \[
    \ti\sigma\colon \left\{\begin{array}{rcll}
      x &\mapsto& x^\sigma, &\textrm{for each } x \in \X,\\
      y &\mapsto& w_y(x_1,\ldots,x_m)^\sigma, &\textrm{for each } y \in {\mc Y}.\\
    \end{array}\right.
  \]
  Then, by Proposition~\ref{prop:NTAddGensIt}, a finite $L$-presentation for 
  the group $G$ is given by 
  \[
    \la \X \cup {\mc Y} \mid \emptyset \mid \{ \ti\sigma \}_{\sigma \in \Phi} \mid 
    \R \cup \{ y^{-1} w_y(x_1,\ldots,x_m)\}_{ y \in {\mc Y} } \ra.
  \]
  As this $L$-presentation is invariant, every $\ti\sigma$, with
  $\sigma \in \Phi$, induces an endomorphism of the group $G$. Thus,
  as $x =_G w_x(y_1,\ldots,y_n)$ holds, we have 
  $x^{\ti\sigma} =_G w_x(y_1,\ldots,y_n)^{\ti\sigma}$ for each $\sigma \in
  \Phi^*$. By Proposition~\ref{prop:NTAddItRel}, we have that
  \begin{equation}
    G \cong \la \X \cup {\mc Y} \mid \emptyset \mid \{ \ti\sigma \}_{\sigma \in \Phi} \mid 
    \R \cup \{ y^{-1} w_y\}_{ y \in {\mc Y} } \cup 
    \{x^{-1} w_x \}_{x \in \X} \ra. \Label{eqn:InvGenSet1}
  \end{equation}
  Since ${\mc Y}$ generates $H$, for each $z \in \X\cup {\mc Y}$
  and $\sigma \in \Phi$, the image $z^{\ti\sigma}$ is represented by
  a word $v_{z,\sigma}(y_1,\ldots,y_n)$ over the generators ${\mc Y}$
  so that $z^{\ti\sigma} =_G v_{z,\sigma}(y_1,\ldots,y_n)$ holds. Since
  the $L$-presentation in Eq.~(\ref{eqn:InvGenSet1}) is invariant,
  Proposition~\ref{prop:NTModSubst} applies to the relation $r = (z^{\ti\sigma})^{-1}
  v_{z,\sigma}(y_1,\ldots,y_n)$ and it shows that $G$ admits the following
  finite $L$-presentation
  \[
    \la \X\cup{\mc Y}\mid \emptyset \mid \{ \widehat\sigma \}_{\sigma \in \Phi} 
    \mid \R \cup \{x^{-1}w_x\}_{x\in\X} \cup \{y^{-1}w_y\}_{y\in{\mc Y}} \cup 
    \{ (z^{\ti\sigma})^{-1} v_{z,\sigma} \}_{z\in\X\cup{\mc Y},\sigma\in\Phi} \ra
  \]
  where the substitutions $\widehat\sigma$ are induced by the maps
  \[
    \widehat\sigma\colon z \mapsto v_{z,\sigma}(y_1,\ldots,y_n),\textrm{ for each }z\in\X\cup{\mc Y}.
  \]
  We use the iterated relations $x^{-1} w_x(y_1,\ldots,y_n)$, with $x \in \X$,
  to replace every occurrence of $x \in \X$ among the iterated relations
  \begin{equation}
    \R \cup \{ y^{-1} w_y(x_1,\ldots,x_m) \}_{y \in {\mc Y}} \cup \{
    (z^{\ti\sigma})^{-1} v_{z,\sigma}(y_1,\ldots,y_n)\}_{z\in\X\cup{\mc Y},\sigma\in\Phi
    } \Label{eqn:InvGenSet2}
  \end{equation}
  by $w_x(y_1,\ldots,y_n)$. This yields a finite set of relations
  $\ti\S$ that can be considered as a finite subset of the free
  group over the alphabet ${\mc Y}$. Replacing the relations
  in Eq.~(\ref{eqn:InvGenSet2}) by $\ti \S$ does not change the
  isomorphism type of the group. The group $G$ satisfies that $G \cong \la \X \cup {\mc Y}
  \mid \emptyset \mid \{\widehat\sigma\mid\sigma\in\Phi\}\mid \ti\S \cup \{x^{-1} w_x \mid x
  \in \X\}\ra$. By Proposition~\ref{prop:NTAddGensIt}, the group $G$
  is invariantly finitely $L$-presented by $\la {\mc Y} \mid \emptyset
  \mid \{\widehat\sigma\}_{\sigma\in\Phi} \mid \ti \S\ra$.
\end{CenProof}

\section{Finitely generated normal subgroups of finitely presented groups}
\Label{sec:FGNorOfFP}
In this section, we consider finitely generated \emph{normal} subgroups of
finitely presented groups.
By Higman's embedding theorem~\cite{Hig61}, every finitely generated
group embeds into a finitely presented group if and only if it is
recursively presented. This theorem classifies the finitely generated subgroups
of a finitely presented group. The \emph{normal}
subgroups of a finitely presented group are invariantly $L$-presented:
\begin{proposition}\Label{prop:NorSubsFpGroups}
  Every normal subgroup of a finitely presented group admits an invariant
  $L$-presentation whose substitutions induce automorphisms of the
  subgroup. If the normal subgroup has finite index, it is invariantly
  finitely $L$-presented.
\end{proposition} 
\begin{proof}
  This follows from the proof of~\cite[Theorem~6.1]{Har11b}; cf.
  Lemma~\ref{lem:FirstObserv} below.
\end{proof}
The $L$-presentation in Lemma~\ref{lem:FirstObserv} below is an ascending
$L$-presentation with finitely many substitutions and finitely many
iterated relations. It has finitely many generators if and only if the
subgroup has finite index. The substitutions of this $L$-presentation
induce automorphisms of the subgroup since they are induced by conjugation
in the finitely presented group.\smallskip

On the other hand, as every finite $L$-presentation is recursive,
finitely $L$-presented groups embed into finitely presented groups. As
indicated in~\cite{Ben11}, a finitely $L$-presented group embeds as
a normal subgroup into a finitely presented group if we assume that
every substitution of the $L$-presentation induces an automorphism of
the subgroup:
\begin{proposition}\Label{prop:InvLpEmbAsNor}
  Every group that admits an invariant finite $L$-presentation, whose
  substitutions induce automorphisms of the group, embeds as a normal
  subgroup into a finitely presented group.
\end{proposition}
\begin{proof}
  If $H = \la{\mc Z} \mid \emptyset \mid \{\delta_1,\ldots,\delta_n\}
  \mid \R\ra$ is invariantly finitely $L$-presented so that each
  $\delta_i$ induces an automorphism of $H$, the base group $H$
  embeds into the HNN-extension $G_1$ relative to the isomorphism
  $\delta_1\colon H\to H$ which is induced by the substitution $\delta_1$. The
  HNN-extension $G_1$ is given by the presentation $G_1 = \la {\mc
  Z} \cup \{t_1\} \mid \bigcup_{\sigma\in\Phi^*} \R^\sigma \cup \{
  t_1^{-1} z t_1 = z^{\delta_1} \mid z\in {\mc Z} \} \ra$ where $\Phi
  = \{\delta_1,\ldots,\delta_n\}$. Denote by $H_1$ the image of $H$
  in $G_1$. Then $\delta_2$ induces an automorphism of the subgroup
  $H_1 \leq G_1$. Thus we can form the HNN-extension $G_2$ relative
  to the isomorphism $\delta_2\colon H_1 \to H_1$.  As the base group
  $G_1$ embeds into the HNN-extension $G_2$, the subgroup $H_1$ embeds
  into $G_2$ as well. Iterating this process, we obtain a group $G_n = \la {\mc Z}
  \cup \{t_1,\ldots,t_n\} \mid \bigcup_{\sigma\in\Phi^*} \R^\sigma \cup
  \{ t_i^{-1} z t_i = z^{\delta_i} \mid 1\leq i\leq n\}\ra$ in which $H$ embeds.  Tietze
  transformations that replace every $\delta_i$-image $z^{\delta_i}$
  by the conjugate $t_i^{-1} z t_i$ in the relations $\bigcup_{\sigma\in\Phi^*} \R^\sigma$ eventually show that $G_n = \la
  {\mc Z} \cup \{t_1,\ldots,t_n\} \mid \R \cup \{ t_i^{-1} z t_i =
  z^{\delta_i} \mid 1\leq i\leq n,z \in {\mc Z}\}\ra$ is finitely
  presented.  The invariantly finitely $L$-presented group $H$ embeds
  into this finitely presented group by identifying the generator in ${\mc
  Z}$. The image of $H$ in $G_n$ is obviously a normal subgroup of $G_n$.
\end{proof}
In the following, we use the constructions from~\cite{Ben11} to prove
Theorem~\ref{thm:GSplits}. Since every normal subgroup of a finitely
presented group admits an invariant $L$-presentation with finitely many
substitutions and finitely many iterated relations, it suffices
to show that the $L$-presentation in Lemma~\ref{lem:FirstObserv}
below could be transformed into an invariant finite $L$-presentation. For
this purpose, though, we need to eliminate (possibly) infinitely many
generators from the $L$-presentation and we need to modify finitely many
substitutions. However, Proposition~\ref{prop:NTModSubst} adds iterated
relations for each modification of a substitution. Hence, we need to
ensure that this process still gives a \emph{finite} $L$-presentation.
In the following, we generalize the constructions from~\cite{Ben11}:

\subsection{Preliminaries}\Label{sec:PrelFGN}
Let $G$ be a finitely
presented group and let $H\unlhd G$ be a finitely generated
normal subgroup. Then $G/H$ is finitely presented. Moreover,
if $H = \la a_1,\ldots,a_m \ra$ and $G / H = \la s_1H,\ldots,s_n
H \ra$ hold, there exists a finite presentation $
\la\,\{a_1,\ldots,a_m,s_1,\ldots,s_n\} \mid \R\,\ra$ for $G$. The proof
of~\cite[Theorem~6.1]{Har11b} yields the following
\begin{lemma}\Label{lem:FirstObserv}
  Let $\la\,\{a_1,\ldots,a_m,s_1,\ldots,s_n\}\mid \R\,\ra$ be a finite
  presentation for $G$ and write $\S = \{s_1^{\pm 1},\ldots,s_n^{\pm 1}\}$. 
  If ${\mc T}$ is a Schreier transversal for $H
  = \la a_1,\ldots,a_m \ra$ in $G$ and ${\mc Y}$ are the Schreier
  generators of $H$, then $H$ is invariantly $L$-presented by
  $$
    \la {\mc Y} \mid \emptyset \mid \{ \delta_x \mid x \in \S \} 
    \mid \R^\tau \ra
  $$
  where $\delta_x$ denotes the endomorphism of the free group $F({\mc
  Y})$ that is induced by conjugation with $x \in \S$ and $\tau$ denotes
  the Reidemeister-rewriting.
\end{lemma}
\begin{proof}
  This follows from the Reidemeister-Schreier
  theorem, see~\cite[Section~II.4]{LS77} and the proof
  of~\cite[Theorem~6.1]{Har11b}. Clearly, one can always omit the
  endomorphisms $\delta_x$ with $x \in \{a_1,\ldots,a_m\}$ as they give
  inner automorphisms of the subgroup $H$.
\end{proof}
Since $\S$ and $\R$ are finite, the $L$-presentation in
Lemma~\ref{lem:FirstObserv} is finite if and only if $H$ has finite
index in $G$; in this case ${\mc Y}$ is finite. Finite index subgroups
of finitely $L$-presented groups have been studied in~\cite{Har11b}. It
was shown that each normal subgroup of a finitely presented group with
finite index is invariantly finitely $L$-presented. In the following,
we therefore assume that $[G:H] = \infty$ holds.\smallskip

The strategy in the proof of Theorem~\ref{thm:GSplits} will be as follows:
Our choice of the generating set of the finitely presented group allows
us to assume that $H$'s generators ${\mc Z} = \{a_1,\ldots,a_m\}$ are
Schreier generators of $H$. We therefore obtain an embedding $\chi\colon
F({\mc Z}) \to F({\mc Y})$ and we will construct an epimorphism
$\gamma\colon F({\mc Y}) \to F({\mc Z})$ so that the free presentation
$\pi\colon F({\mc Y}) \to H$ that is given by the $L$-presentation in
Lemma~\ref{lem:FirstObserv} satisfies $\gamma\chi\pi = \pi$. Since
the $L$-presentation in Lemma~\ref{lem:FirstObserv} is invariant,
there exists, for each $\sigma \in \Phi = \{\delta_x \mid x\in \S\}$,
an endomorphism $\widehat\sigma \in \End(H)$ so that $\sigma\pi =
\pi\widehat\sigma$ holds. In general, we cannot assume that there
also exists an endomorphism $\ti\sigma \in \End(F({\mc Z}))$ so that
$\sigma\gamma = \gamma\ti\sigma$ holds. Therefore, we will construct a
normal subgroup $N \unlhd F({\mc Z})$ so that $\psi \colon F({\mc Z}) \to
F({\mc Z}) / N,\: g \mapsto gN$ yields the existence of $\bar\sigma \in
\End(F({\mc Z})/N)$ with $\sigma\gamma\psi = \gamma\psi \bar\sigma$. These
constructions will give the following commutative diagram:
\[
  \xymatrix{ 
  & F({\mc Z}) \ar[r]^\psi \ar[d]^{\chi\pi} & F({\mc Z}) / N \ar[dl] \ar@(ul,ur)^{\bar \delta_x} \\
  F({\mc Y}) \ar@(dl,dr)_{\delta_x} \ar[ur]^\gamma \ar[r]^\pi & H \ar@(dl,dr)_{\widehat\delta_x} & 
  }
\]
In the special cases of Theorem~\ref{thm:GSplits} and
Theorem~\ref{thm:TorRk2}, we are able to prove that $F({\mc Z}) / N$
is invariantly finitely $L$-presented and so is the subgroup $H$.
The normal subgroup $N$ will be generated, as a normal subgroup,
by the iterated relations that Proposition~\ref{prop:NTModSubst}
adds when modifying the substitutions of the $L$-presentation
in Lemma~\ref{lem:FirstObserv}. These relations were omitted
in~\cite{Ben11}. It is not clear whether or not these relations are
necessary to define the subgroup $H$.\medskip

In the remainder of this section, we generalize the constructions from~\cite{Ben11} to obtain the commutative diagram above.
The generating set $\X = \{a_1,\ldots,a_m,s_1,\ldots,s_n\}$ of the
finitely presented group $G$ yields that the generators ${\mc Z}
= \{a_1,\ldots,a_m\}$ are Schreier generators of $H$.  Hence, there
exists a natural embedding $\chi\colon F({\mc Z}) \to F({\mc Y})$ which is
induced by embedding the generators ${\mc Z}$ into ${\mc Y}$. 
It suffices to remove the Schreier generators
${\mc Y} \setminus {\mc Z}$ from the invariant $L$-presentation in
Lemma~\ref{lem:FirstObserv}.
Since $H$ is generated by ${\mc Z} = \{a_1,\ldots,a_m\}$, every $y \in
{\mc Y}$ can be represented, as an element of $H$, by a word over ${\mc
Z}$. This yields an epimorphism $\gamma\colon F({\mc Y}) \to F({\mc Z})$
which maps every $y \in {\mc Y}$ to a word $y^\gamma \in F({\mc Z})$
over the alphabet ${\mc Z}$ that represents the same element in $H$;
i.e., we have 
\begin{equation}\Label{eqn:IotaPi}
  \{ y^{-1} y^{\gamma\chi} \mid y \in {\mc Y} \setminus {\mc Z} \} 
  \subseteq \ker(\pi),
\end{equation}
where $\pi\colon F({\mc Y}) \to H$ denotes the free presentation
from Lemma~\ref{lem:FirstObserv}.  Note that Eq.~(\ref{eqn:IotaPi})
yields that $\iota = \chi\pi$ defines an epimorphism $\iota\colon
F({\mc Z}) \to H$ with $\gamma\iota = \pi$. The following lemma
generalizes~\cite[Lemma~4]{Ben11}.
\begin{lemma}\Label{lem:EpiPres}
  If $H \cong \la {\mc Y} \mid {\mc S} \ra$ and $\gamma\colon F({\mc Y})
  \to F({\mc Z})$ is an epimorphism so that
  \[
    \xymatrix{& F({\mc Z}) \ar[d]^\iota \\
    	      F({\mc Y}) \ar[ur]^\gamma \ar[r]_\pi & H }
  \]
  commutes, $\la{\mc Z} \mid \S^\gamma \ra$ is a presentation for $H$.
\end{lemma}
\begin{proof}
  Since $\pi = \gamma\iota$ is onto, it suffices to prove that
  $\ker(\iota) = \la \S^\gamma \ra^{F({\mc Z})}$ holds. For $r \in
  \S$, we have that $r^{\gamma\iota} = r^\pi = 1$ and so $r^\gamma
  \in \ker(\iota)$. Thus $\la \S^\gamma \ra^{F({\mc Z})} \subseteq
  \ker(\iota)$. If $g \in \ker(\iota)$ holds, there exists $h \in F({\mc
  Y})$ with $h^\gamma = g$ as $\gamma$ is surjective. Then $h^\pi =
  h^{\gamma\iota} = g^\iota = 1$ and $h \in \ker(\pi) = \la \S \ra^{F({\mc
  Z})}$. Thus $g = h^\gamma \in \la \S^\gamma \ra^{F({\mc Z})}$.
\end{proof}
Thus, by Lemma~\ref{lem:FirstObserv} and Lemma~\ref{lem:EpiPres}, the
subgroup $H$ has a presentation of the form
\[
  H =  \big\la {\mc Z} \:\big|\: 
  \{ (r^{\tau \sigma})^\gamma \mid r\in \R,\: \sigma \in  \Phi^* \} \big\ra
\]
where $\Phi = \{ \delta_x \mid x \in \S \}$ and $\tau$ denotes the
Reidemeister rewriting. This presentation can be considered as a
finite $L$-presentation if, for each $\sigma \in \Phi$, there exists
an endomorphism $\ti\sigma \in \End(F({\mc Z}))$ with $\sigma\gamma =
\gamma\ti\sigma$. The following lemma yields the existence of such
endomorphisms $\ti\sigma \in \End(F({\mc Z}))$:
\begin{lemma}\Label{lem:ComLem}
  For groups $L$ and $M$, an epimorphism $\pi \colon L \to M$, and an
  endomorphism $\delta \in \End(L)$, there exists a (unique) endomorphism
  $\Delta \in \End(M)$ with $\delta\pi = \pi\Delta$
  if and only if $\ker(\pi)^\delta \subseteq \ker(\pi)$ holds.
\end{lemma}
\begin{proof}
  The proof is straightforward.
\end{proof}
Therefore, if the kernel $\ker(\gamma)$ is $\sigma$-invariant, for
each $\sigma\in\Phi$, the subgroup $H$ would be invariantly finitely
$L$-presented by $\la {\mc Z} \mid \emptyset \mid \{\ti\delta_x \mid
\delta_x\in\Phi\} \mid \R^{\tau\gamma} \ra$.  In general, though, we
cannot assume that each $\sigma\in\Phi$ leaves the kernel $\ker(\gamma)$
invariant. If we consider the natural embedding $\chi\colon F({\mc Z})
\to F({\mc Y})$ that is induced by embedding the generators ${\mc Z}$
into ${\mc Y}$, the kernel $\ker(\gamma)$ satisfies
\begin{lemma}\Label{lem:ChiGamma}
  If $\chi\colon F({\mc Z}) \to F({\mc Y})$ is an embedding with
  $\gamma\chi|_{{\mc Z}} = \id_{{\mc Z}}$, then $\chi\gamma = \id_{F({\mc Z})}$
  and $\ker(\gamma) = \la \{ y^{-1} y^{\gamma\chi} \mid y\in {\mc Y}
  \setminus {\mc Z} \} \ra^{F({\mc Y})}$ hold.
\end{lemma}
\begin{proof}
  Since $\gamma\chi |_{\mc Z} = \id_{\mc Z}$ holds, the map $\gamma\chi$
  induces the identity on the free subgroup $E = \la {\mc Z} \ra \leq
  F({\mc Y})$. For $g \in F({\mc Z})$, we have $g^\chi \in E$ and
  $g^{\chi\gamma\chi} = g^\chi$. Thus $(g^{-1} g^{\chi\gamma})^\chi =
  1$ and, as $\chi$ is injective, we have $g^{-1} g^{\chi\gamma} = 1$ or
  \begin{equation} \Label{eqn:ChiGammaId}
    \chi\gamma = \id_{F({\mc Z})}.
  \end{equation}
  For each $y\in {\mc Y}\setminus {\mc
  Z}$, we have that $(y^{-1} y^{\gamma\chi})^\gamma = y^{-\gamma}
  y^{\gamma\chi\gamma} = y^{-\gamma} y^\gamma = 1$. Therefore $N =
  \la \{ y^{-1} y^{\gamma\chi} \mid y\in {\mc Y} \setminus {\mc Z} \}
  \ra^{F({\mc Y})}$ satisfies that $N \subseteq \ker(\gamma)$.  Let $g
  \in \ker(\gamma)$ be given. Then $g \in F({\mc Y})$ is represented by
  a finite word $w(y_{i_1},\ldots,y_{i_n},a_1,\ldots,a_m)$ with 
  $\{y_{i_1},\ldots,y_{i_n}\} \subseteq {\mc Y} \setminus {\mc Z}$.
  Modulo the normal subgroup $N$, we can replace every
  occurrence of $y \in {\mc Y} \setminus {\mc Z}$ by $y^{\gamma\chi}
  \in E$; i.e., we have $g = w(y_{i_1},\ldots,y_{i_n},a_1,\ldots,a_m) =
  w(y_{i_1}^{\gamma\chi},\ldots,y_{i_n}^{\gamma\chi},a_1,\ldots,a_m) \cdot h$
  for some $h \in N$. As $g \in \ker(\gamma)$ and $h \in N \subseteq
  \ker(\gamma)$ hold, we have 
  \[
    1 = g^{\gamma\chi}
    = w(y_{i_1}^{\gamma\chi\gamma\chi},\ldots,y_{i_n}^{\gamma\chi\gamma\chi},a_1^{\gamma\chi},\ldots,a_m^{\gamma\chi}) \cdot h^{\gamma\chi} 
    = w(y_{i_1}^{\gamma\chi},\ldots,y_{i_n}^{\gamma\chi},a_1,\ldots,a_m) \cdot 1.
  \]
  Similarly, modulo the normal subgroup $N$, we can replace every occurrence
  of $y^{\gamma\chi}$ by $y$. There exists $k \in N$ with $1 =
  w(y_{i_1}^{\gamma\chi},\ldots,y_{i_n}^{\gamma\chi},a_1,\ldots,a_m) =
  w(y_{i_1},\ldots,y_{i_n},a_1,\ldots,a_m)\cdot\nolinebreak k = g \cdot k$. Thus
  $g \in N$ and $N = \ker(\gamma)$.
\end{proof}
Even though $\delta_x \in \Phi$ may not translate directly to $\ti\delta_x
\in \End(F({\mc Z}))$, there exists a normal subgroup $N \unlhd F({\mc
Z})$ and a homomorphism $\psi\colon F({\mc Z}) \to F({\mc
Z}) / N, \: g\mapsto gN$ so that $\ker(\gamma\psi)^{\delta_x} \subseteq
\ker(\gamma\psi)$ holds: For each $\delta_x \in \Phi$, define $\ti\delta_x
= \chi\delta_x\gamma \in \End(F({\mc Z}))$.  Consider the normal subgroup
\begin{equation}\Label{eqn:NormSubFac}
  N = \Big\la 
  \bigcup_{\sigma \in \ti\Phi^* } \Big(
  \big\{ (y^{-1} y^{\gamma\chi})^{\delta_x\gamma} \big\}_{
          y \in {\mc Y} \setminus {\mc Z}, x \in \S}\Big)^\sigma
  \Big\ra^{F({\mc Z})}
\end{equation}
where $\ti\Phi = \{ \ti\delta_x \mid \delta_x \in \Phi \}$. By
construction, $N$ satisfies $N^{\ti\delta_x} \subseteq N$ and thus there
exists a unique endomorphism $\bar\delta_x\colon F({\mc Z}) / N \to
F({\mc Z}) / N, \: gN \mapsto g^{\ti\delta_x} N$ with $\ti\delta_x \psi
= \psi \bar\delta_x$. The normal subgroup $N$ allows us to translate
$\delta_x \in \Phi$ to $\bar\delta_x \in \End(F({\mc Z})/N)$ with
$\delta_x \gamma\psi = \gamma\psi \bar\delta_x$:
\begin{lemma}
  For each $x \in \S$, we have that
  $\ker(\gamma\psi)^{\delta_x} \subseteq \ker(\gamma\psi)$.
\end{lemma}
\begin{proof}
  The kernel $\ker(\gamma\psi) = \ker(\gamma)\,N^{\gamma^{-1}}$ satisfies that 
  \begin{eqnarray*}
    \ker(\gamma\psi) &=& \Big\la
    \Big\{ y^{-1} y^{\gamma\chi} \Big\}_{y \in {\mc Y}\setminus {\mc Z}}
    \cup \bigcup_{\ti\sigma \in \Phi^*}
    \Big\{ (y^{-1} y^{\gamma\chi})^{\delta_z\gamma\ti\sigma\chi} \Big\}_{
    y \in {\mc Y} \setminus {\mc Z} \atop z \in \S }
    \Big\ra^{F({\mc Y})}.
  \end{eqnarray*}
  The generator $(y^{-1}
  y^{\gamma\chi})^{\delta_z\gamma\ti\sigma\chi}$
  is mapped by $\delta_x\gamma$ to $(y^{-1}
  y^{\gamma\chi})^{\delta_z\gamma\ti\sigma\chi\delta_x\gamma} =
  (y^{-1} y^{\gamma\chi})^{\delta_z\gamma\ti\sigma\ti\delta_x}
  \in N$ while $y^{-1} y^{\gamma\chi}$ is mapped to $(y^{-1}
  y^{\gamma\chi})^{\delta_x\gamma} \in N$.
\end{proof}
The endomorphisms $\delta_x \in \End(F({\mc Y}))$, $\ti\delta_x \in
\End(F({\mc Z}))$, and $\bar\delta_x \in \End(F({\mc Z})/N)$ also
satisfy that
\begin{equation}
  \ti\delta_x\psi =  \chi\delta_x\gamma\psi = \chi\gamma\psi\bar\delta_x
  = \psi\bar\delta_x. 
\end{equation}
Since the $L$-presentation in Lemma~\ref{lem:FirstObserv} is invariant,
there exists $\widehat\delta_x\in \End(H)$ with $\delta_x \pi =
\pi \widehat\delta_x$. The subgroup $H$ is a homomorphic image of 
$F({\mc Z}) / N$:
\begin{lemma}\Label{lem:NorIota}
  Let $\iota\colon F({\mc Z}) \to H, g \mapsto g^{\chi\pi}$ be given.
  Then $\gamma\iota = \pi$ and $N \leq \ker(\iota)$.
\end{lemma}
\begin{proof}
  The first assertion follows from the definition of $\gamma$ in
  Eq.~(\ref{eqn:IotaPi}) above.  For $\delta_x \in \Phi$, we have $\ti\delta_x
  \iota = \chi\delta_x\gamma\iota = \chi\delta_x \pi = \chi \pi
  \widehat\delta_x = \iota \widehat\delta_x$ and $\gamma\chi\pi =
  \gamma\iota = \pi$. Thus $(y^{-1} y^{\gamma\chi})^\pi = y^{-\pi}
  y^{\gamma\chi\pi} = y^{-\pi} y^\pi  = 1$. For $\ti\sigma \in \ti\Phi^*$
  with $\ti\sigma = \ti\delta_{x_1} \cdots \ti\delta_{x_n}$ we therefore obtain
  \[
    \delta_x \gamma \ti\sigma \iota = 
    \delta_x\gamma\ti\delta_{x_1} \cdots \ti\delta_{x_n} \iota =
    \delta_x\gamma\iota \widehat\delta_{x_1} \cdots \widehat\delta_{x_n} =
    \delta_x\pi \widehat\delta_{x_1} \cdots \widehat\delta_{x_n} =
    \pi\widehat \delta_x \widehat\delta_{x_1} \cdots \widehat\delta_{x_n}.
  \]
  Hence, for each $\ti\sigma \in \Phi^*$, $y \in {\mc Y}
  \setminus {\mc Z}$, and $x\in\X$, the generator
  $(y^{-1}y^{\gamma\chi})^{\delta_x\gamma\ti\sigma} \in N$ satisfies
  $(y^{-1}y^{\gamma\chi})^{\delta_x\gamma \ti\sigma \iota} = (y^{-1}
  y^{\gamma\chi})^{\pi\widehat\delta_x\,\widehat\delta_{x_1} \cdots
  \widehat\delta_{x_n}} = 1$ as $y^{-1} y^{\gamma\chi} \in \ker(\pi)$
  holds. Therefore $N \subseteq \ker(\iota)$ holds.
\end{proof}
By Lemma~\ref{lem:NorIota}, the homomorphism $\varphi\colon F({\mc Z})
/ N \to H,\: gN \mapsto g^\iota$ is well-defined and it satisfies that
$\psi\varphi = \iota$. We have obtained the following diagram:
\[
  \xymatrix{ & F({\mc Z}) \ar@/^/[dl]^\chi \ar@(ul,ur)^{\ti\delta_x} \ar[d]^{\iota=\chi\pi} \ar[r]^\psi & F({\mc Z}) / N \ar@(ul,ur)^{\bar\delta_x} \ar[dl]^\varphi \\
  F({\mc Y}) \ar@/^/[ur]^\gamma \ar@(dl,dr)_{\delta_x} \ar[r]_\pi & H \ar@(dl,dr)_{\widehat\delta_x}& }
\]
By construction, $F({\mc Z}) / N$ is invariantly $L$-presented by
\[
  F({\mc Z}) / N \cong 
  \la {\mc Z} \mid \emptyset \mid
  \{ \ti\delta_x \}_{\delta_x\in\Phi} \mid
  \{ (y^{-1}y^{\gamma\chi})^{\delta_x\gamma} \}_{y\in {\mc Y}
  \setminus {\mc Z}, \delta_x \in \Phi} \ra.
\]
If $[G:H] = \infty$ holds, $|{\mc Y} \setminus {\mc Z}|$ is
infinite. Therefore, the latter $L$-presentation is finite if
and only if $[G:H]$ is finite. Our strategy in the proof of
Theorem~\ref{thm:GSplits} uses the following
\begin{lemma}\Label{lem:FZNFac}
  If there exists a finite set ${\mc U} \subseteq F({\mc Z})$ with
  $F({\mc Z}) / N \cong \la {\mc Z} \mid \emptyset \mid \ti\Phi \mid
  {\mc U} \ra$, then $H$ is invariantly finitely $L$-presented.
\end{lemma}
\begin{proof}
  The kernel of $\varphi\colon F({\mc Z}) / N \to H$ is generated by
  the images $r^{\tau\sigma\gamma\psi} = r^{\tau\gamma\psi \bar\sigma}$
  with $\sigma \in \Phi^*$ and $r\in \R$. If $\la {\mc Z} \mid \emptyset
  \mid \ti\Phi \mid {\mc U} \ra$ is an invariant finite $L$-presentation for
  $F({\mc Z}) / N$, then $H$ is invariantly finitely $L$-presented by
  $\la{\mc Z}\mid\emptyset\mid\ti\Phi\mid {\mc U} \cup \R^{\tau\gamma} \ra$.
\end{proof}

\subsection{Proofs of Theorem~\ref{thm:GSplits} and Theorem~\ref{thm:TorRk2}}\Label{sec:Proofs}
In this section, we prove Theorem~\ref{thm:GSplits} and
Theorem~\ref{thm:TorRk2}: 
\def\0{Theorem~\ref{thm:GSplits}}
\begin{CenProof}{ of \0}
  Our strategy in the proof of Theorem~\ref{thm:GSplits} is to construct a
  normal subgroup $N\unlhd F({\mc Z})$ and to prove that $F({\mc Z}) / N$
  is invariantly finitely $L$-presented.  Then Lemma~\ref{lem:FZNFac}
  applies and it shows that $H \leq G$ is invariantly finitely
  $L$-presented.\smallskip

  Since $G$ is finitely presented, $G / H$ is finitely
  generated. Moreover, as $G$ splits over $H$, there exists
  $s_1,\ldots,s_n \in G$ so that $G / H = \la s_1H,\ldots,s_n H \ra$
  and $S = \la s_1,\ldots,s_n\ra$ satisfies that $S \cap H = \{1\}$;
  i.e., $G \cong H \rtimes S$ holds. Because $H$ is finitely generated,
  there exist $a_1,\ldots,a_m \in H$ so that $H = \la a_1,\ldots,a_m
  \ra$ holds. Then $G = \la a_1,\ldots,a_m, s_1,\ldots,s_n \ra$ holds and there
  exists a finite set of relations $\R$ with $G \cong \la\,
  \{a_1,\ldots,a_m, s_1,\ldots,s_n\} \mid \R\,\ra$. Write $\S
  = \{s_1^{\pm 1},\ldots,s_n^{\pm 1}\}$ and $\X = \{a_1,\ldots,a_m,
  s_1,\ldots,s_n \}$. Clearly, we can choose a Schreier transversal
  ${\mc T} \subseteq \S^*$ whose elements are words over the alphabet
  $\S$.  This yields the Schreier generators
  \begin{eqnarray*}
    a_{\ell,t} = \gamma(t,a_\ell)
    &=& t a_\ell (\overline{t a_\ell})^{-1} = t a_{\ell} t^{-1},\\
    s_{\ell,t} = \gamma(t,s_\ell)
    &=& t s_\ell (\overline{ts_\ell})^{-1},
  \end{eqnarray*}
  with $t \in {\mc T}$. Then $\{ s_{\ell,t} \mid 1\leq\ell\leq n, t
  \in {\mc T} \} \subseteq \S^{*}$. By Lemma~\ref{lem:FirstObserv},
  the subgroup $H$ is invariantly $L$-presented by $\la{\mc Y} \mid
  \emptyset \mid \{ \delta_s \mid s \in \S \} \mid \R^\tau \ra$ where
  \[
    {\mc Y} = \{ a_{\ell,t} \mid t\in {\mc T}, 1\leq\ell\leq m \}
    \cup \{ s_{\ell,t} \neq 1\mid  t\in {\mc T}, 1\leq\ell\leq n \}
  \]
  and $\delta_s$ denotes the endomorphism of $F({\mc Y})$ that is induced
  by conjugation with $s \in \S$.  Write $S = \la s_1,\ldots,s_n \ra\leq
  F({\mc X})$ and $E = \la a_1,\ldots,a_m \ra \leq F({\mc X})$. Let
  $K \unlhd F({\mc X})$ be the kernel of $G$'s free presentation $F({\mc
  X}) \to G$. Then $E\!K = \la {\mc Y} \ra$ and $S\cap
  E\!K = \la s_{\ell,t} \neq 1 \mid 1\leq\ell\leq n, t\in {\mc T}\ra$
  are freely generated. For each $s \in \S$, the subgroup $S \cap E\!K$
  is $\delta_s$-invariant since $S \cap E\!K \unlhd S$ holds. Because $G$
  splits over $H$, we have $S \cap H = \{1\}$. Thus the 
  generators $s_{\ell,t} \in S \cap E\!K$ are contained in the kernel of the
  free presentation $\pi\colon F({\mc Y}) \to H$ which is given by $H$'s
  invariant $L$-presentation above. Define ${\mc Z} = \{a_1,\ldots,a_m\}$
  and an embedding
  \[
    \chi\colon F({\mc Z}) \to F({\mc Y}),\: a_\ell \mapsto a_{\ell,1}
  \]
  where $1 \in {\mc T}$ denotes the trivial element in the
  Schreier transversal ${\mc T}$. For $s \in \S$ and $a_\ell \in
  {\mc Z}$, we choose a representative $a_\ell^{\chi\delta_s\gamma} \in F({\mc
  Z})$ with
  \begin{equation}\Label{eqn:Split01}
    a_\ell^{-\chi\delta_s}\, (a_\ell^{\chi\delta_s\gamma})^\chi
    \in \ker(\pi).
  \end{equation}
  For $s\in\S$, let $\ti\delta_s \in F({\mc Z})$ be 
  induced by the map  $a_\ell \mapsto a_\ell^{\chi\delta_s\gamma}$ and define
  $\iota\colon F({\mc Z}) \to H$ by $\iota = \chi\pi$. Then
  Eq.~(\ref{eqn:Split01}) yields that $\ti\delta_s \iota =
  \iota\widehat\delta_s$. In the following, we write $\ti\delta_t =
  \ti\delta_{x_1}\cdots \ti\delta_{x_n}$ if $t = x_1 \cdots x_n \in \S^*$
  and each $x_i \in \S$. Moreover, we write $X$ for $x^{-1}$ and
  $T$ for $t^{-1}$. This yields that $a_{\ell,1} ^{\delta_T} =
  ta_\ell T = a_{\ell,t}$. Let $\gamma\colon F({\mc Y}) \to F({\mc
  Z})$ be induced by the map
  \[
    \gamma\colon \left\{ \begin{array}{rcll}
    a_{\ell,t} &\mapsto& a_{\ell}^{\ti\delta_T}, &\textrm{for each }
    1 \leq \ell \leq m\textrm{ and } t \in {\mc T},\\
    s_{\ell,t} &\mapsto& 1,&\textrm{for each }1\leq\ell\leq n\textrm{ and }t \in {\mc T}.
    \end{array}\right.
  \]
  For each $1\leq\ell\leq m$, $1\leq k\leq n$, and $t \in {\mc T}$,
  this yields
  \[
    (a_{\ell,t} )^{\gamma\iota} 
    = a_{\ell}^{\ti\delta_T\iota}
    = a_{\ell}^{\iota\widehat\delta_T}
    = a_{\ell}^{\chi\pi\widehat\delta_T}
    = a_{\ell,1}^{\delta_T\pi}
    = a_{\ell,t}^{\pi}
    \quad\textrm{and}\quad
    (s_{k,t})^{\gamma\iota} 
    = 1^\iota = 1 = (s_{k,t})^\pi.
  \]
  Thus $\gamma\iota = \pi$. Define the normal subgroup 
  \[
    N = \Big\la \bigcup_{\ti\sigma \in \ti\Phi^*} \Big(
    \big\{ (y^{-1} y^{\gamma\chi})^{\delta_s \gamma}\big\}_{
    y \in {\mc Y} \setminus {\mc Z}, s\in\S} \Big)^{\ti\sigma}
    \Big\ra^{F({\mc Y})}
  \]
  where $\ti\Phi = \{ \ti\delta_s \mid s\in\S \}$. 
  For $t \in {\mc T}$ and $s \in \S$, we have that 
  \[
    (s_{\ell,t}^{-1} s_{\ell,t}^{\gamma\chi})^{\delta_s\gamma}
    = s_{\ell,t}^{-\delta_s\gamma} (s_{\ell,t}^{\gamma})^{\ti\delta_s} = 1
  \]
  as the subgroup $S \cap E\!K = \la s_{\ell,t} \mid t \in {\mc T},
  1\leq\ell\leq n\ra$ is $\delta_s$-invariant and it is contained
  in the kernel of $\gamma$. This yields that
  \[
    N = \Big\la \bigcup_{\ti\sigma \in \ti\Phi^*}\Big(
    \big\{
    (a_{\ell,t}^{-1} a_{\ell,t}^{\gamma\chi})^{\delta_s \gamma} \big\}_{1\leq\ell\leq m, t \in {\mc T} \setminus \{1\}, s\in\S}\Big)^{\ti\sigma}
    \Big\ra^{F({\mc Z})}.
  \]
  For $t \in {\mc T}$ and $x \in \S$ with $xt \in {\mc T}$, we also
  have that
  \[
    (a_{\ell,t}^{-1} a_{\ell,t}^{\gamma\chi})^{\delta_X\gamma}
    = a_{\ell,t}^{-\delta_X\gamma} a_{\ell,t}^{\gamma\ti\delta_X}
    = a_{\ell,xt}^{-\gamma}\,a_{\ell,t}^{\gamma\ti\delta_X}
    = a_{\ell}^{-\ti\delta_{TX}} a_{\ell}^{\ti\delta_T\ti\delta_X} = 1.
  \]
  It therefore suffices to consider the generators $(a_{\ell,t}^{-1}
  a_{\ell,t}^{\gamma\chi})^{\delta_X\gamma} \in N$ with $1\leq\ell\leq m$,
  $t \in {\mc T}$, and $x \in \S$ but $xt \not\in {\mc T}$. Since $G /
  H \cong S / S\cap E\!K$ is a finitely presented group, there exists
  a finite monoid presentation
  \[
    S / S \cap E\!K \cong \la\:\S \mid (U_1,V_1),\ldots,(U_p,V_p)\:\ra.
  \]
  The monoid congruence $\sim$ induced by this presentation is the
  reflexive, symmetric, and transitive closure of the binary relation
  $\sim$ that is defined by $x \sim y$ if there exist $A,B\in\S^{*}$
  and $1\leq i\leq p$ so that $x = AU_iB$ and $y = AV_iB$ hold. Define
  \[
    M = \Big\la \bigcup_{\ti\sigma\in\ti\Phi^*}\Big(
    \big\{( a_\ell^{-\ti\delta_{U_i}}
    a_\ell^{\ti\delta_{V_i}}) \big\}_{1\leq\ell\leq m, 1\leq i\leq
    p} \Big)^{\ti\sigma}\Big\ra^{F({\mc Z})}.
  \]
  Suppose that $u \sim v$ holds. Then there exist $A_i,B_i,L_i \in \S^{*}$
  so that $u = L_1 \sim \ldots \sim L_q = v$ with $L_i = A_i U_{\ell_i}
  B_i$ and $L_{i+1} = A_i V_{\ell_i} B_i$ (or $L_i = A_i V_{\ell_i} B_i$
  and $L_{i+1} = A_i U_{\ell_i} B_i$). Note that
  \[
    a_\ell^{\ti\delta_{A_i} \ti\delta_{U_{\ell_i}} \ti\delta_{B_i} }
    = (a_\ell^{\ti\delta_{A_i}}) ^{\ti\delta_{U_{\ell_i}} \ti\delta_{B_i} }
    = w_\ell(a_1,\ldots,a_m) ^{\ti\delta_{U_{\ell_i}} \ti\delta_{B_i} }
    = w_\ell(a_1^{\ti\delta_{U_{\ell_i}}},\ldots,a_m^{\ti\delta_{U_{\ell_i}}}) ^{ \ti\delta_{B_i} }
  \]
  for some word $w_\ell(a_1,\ldots,a_m) = a_{\ell}^{\ti\delta_{A_i}}
  \in F({\mc Z})$. The normal subgroup $M$ yields that
  \[
    (a_\ell^{\ti\delta_{A_i}}) ^{\ti\delta_{U_{\ell_i}}} 
    = w_\ell(a_1^{\ti\delta_{U_{\ell_i}}},\ldots,a_m^{\ti\delta_{U_{\ell_i}}})
    = w_\ell(a_1^{\ti\delta_{V_{\ell_i}}},\ldots,a_m^{\ti\delta_{V_{\ell_i}}}) \cdot h
    = a_\ell^{\ti\delta_{A_i}\ti\delta_{V_{\ell_i}}} \cdot h
  \]
  for some $h\in M$. By construction, $M$ is $\ti\Phi^*$-invariant and thus
  \[
    a_\ell^{-\ti\delta_{A_i}
      \ti\delta_{V_{\ell_i}} \ti\delta_{B_i} } a_\ell^{\ti\delta_{A_i}
      \ti\delta_{U_{\ell_i}} \ti\delta_{B_i} } = h^{\ti\delta_{B_i}} \in M.
  \] 
  This shows that, if $u \sim v$ holds, we have $a_\ell^{- \ti\delta_u} a_\ell^{\ti\delta_v} \in M$.
  Suppose that, for $t \in {\mc T}$ and
  $x \in \S$, $xt \not\in {\mc T}$ holds. Then there exists $u = \overline{xt} \in {\mc T}$ with
  $u \sim xt$. Write $U$ for $u^{-1}$. Since $S \cap E\!K \unlhd S$
  holds, there exists $h \in S \cap E\!K \subseteq \ker(\gamma)$
  so that $xt = hu$. This yields that $a_{\ell,t}^{\delta_X} =
  xt\,a_\ell \,TX = hu \, a_\ell\,U h^{-1} = h\,a_{\ell,u}\,h^{-1}$
  and $a_{\ell,t}^{\delta_X\gamma} = a_{\ell,u}^{\gamma} =
  a_\ell^{\ti\delta_U}$.  Since $u \sim xt$ and $U \sim TX$ hold,
  we obtain
  \[
    (a_{\ell,t}^{-1} a_{\ell,t}^{\gamma\chi})^{\delta_X\gamma} 
   = a_{\ell,u}^{-\gamma} a_{\ell}^{\ti\delta_T \ti\delta_X} 
   = a_{\ell}^{-\ti\delta_U} a_{\ell}^{\ti\delta_T \ti\delta_X } \in M.
  \] 
  Thus $N \subseteq M$. It suffices to show that $M \subseteq N$
  holds. Since $M$ and $N$ are both normal subgroups of
  $F({\mc Z})$ and both are $\ti\Phi^*$-invariant, it suffices to
  prove that
  $a_\ell^{-\ti\delta_{U_i}} a_\ell^{\ti\delta_{V_i}} \in N = \ker(\psi)$
  holds. Since $\ti\delta_s \psi = \psi \bar\delta_s$ and $\chi\gamma =
  \id_{F({\mc Z})}$ hold, we have that
  \begin{eqnarray*}
    (a_\ell^{-\ti\delta_{U_i}} a_\ell^{\ti\delta_{V_i}} )^\psi 
    &=& a_\ell^{-\psi\bar\delta_{U_i}} a_\ell^{\psi\bar\delta_{V_i}}
    = a_\ell^{-\chi\gamma\psi\bar\delta_{U_i}} a_\ell^{\chi\gamma\psi\bar\delta_{V_i}}
    = a_{\ell,1}^{-\gamma\psi\bar\delta_{U_i}} a_{\ell,1}^{\gamma\psi\bar\delta_{V_i}} \\
    &=& a_{\ell,1}^{-\delta_{U_i}\gamma\psi} a_{\ell,1}^{\delta_{V_i}\gamma\psi}
    = (a_{\ell,1}^{-\delta_{U_i}} a_{\ell,1}^{\delta_{V_i}})^{\gamma\psi}.
  \end{eqnarray*}
  As $S\cap E\!K \unlhd S$ and ${\mc T} \subseteq S$ hold, there exist
  $h \in S\cap E\!K = \la s_{\ell,t} \mid 1\leq\ell\leq n, t\in {\mc
  T}\ra$ and $t = \overline{U_i^{-1}} \in {\mc T}$ with $U_i^{-1}
  = ht$. Thus $a_{\ell,1}^{\delta_{U_i}} = U_i^{-1}\,a_\ell U_i =
  h\,ta_\ell t^{-1} h^{-1} = h\,a_{\ell,t}\,h^{-1}$. Since $h \in
  \ker(\gamma)$ holds, we obtain  $(a_{\ell,1}^{\delta_{U_i}})^\gamma
  = a_{\ell,t}^\gamma$. Since $U_i \sim V_i$ holds, we also
  have that $\overline{V_i^{-1}} = t$. Similarly, we obtain
  $(a_{\ell,1}^{\delta_{V_i}})^\gamma = a_{\ell,t}^\gamma$. Thus
  $a_{\ell,1}^{-\delta_{U_i}} a_{\ell,1}^{\delta_{V_i}} \in \ker(\gamma)$
  and so $(a_\ell^{-\ti\delta_{U_i}} a_\ell^{\ti\delta_{V_i}} )^\psi =
  1$ or $a_\ell^{-\ti\delta_{U_i}} a_\ell^{\ti\delta_{V_i}} \in N$. Thus
  $M = N$. This shows that that factor group $F({\mc Z}) / N$ is invariantly finitely
  $L$-presented and so is our subgroup $H$.
\end{CenProof}
\noindent Even if $G / H$ is free, the finite $L$-presentation of $F({\mc
Z}) / N$ in the proof of Theorem~\ref{thm:GSplits} contains the relations
of a monoid presentation of the free group. It is not clear whether or
not these relations can be omitted as was done in~\cite{Ben11}. However,
the result in~\cite{Ben11} is a consequence of Theorem~\ref{thm:GSplits}
even if these relations are not redundant:
\def\0{Benli~\cite{Ben11}}
\begin{theorem}[\0]\Label{thm:Ben11}
  Every finitely generated normal subgroup of a finitely presented
  group is invariantly finitely $L$-presented if the quotient is
  infinite cyclic.
\end{theorem}
\begin{proof}
  Since the quotient is free, the finitely presented group
  splits over its finitely generated normal subgroup and thus, by
  Theorem~\ref{thm:GSplits}, the subgroup is invariantly finitely
  $L$-presented.
\end{proof}
Even if the finitely presented group does not split over its finitely
generated subgroup, the subgroup is possibly invariantly finitely
$L$-presented:
\begin{theorem}\Label{thm:Rank2}
  Every finitely generated normal subgroup of a finitely presented group
  is invariantly finitely $L$-presented if the quotient is free
  abelian with rank two.
\end{theorem}
\begin{proof}
  Let $G$ be a finitely presented group and let $H \unlhd G$
  be finitely generated so that $G / H \cong \Z \times \Z$ holds.
  By Lemma~\ref{lem:FZNFac}, it suffices to construct a factor group
  $F({\mc Z}) / N$ which is invariantly finitely $L$-presented. Since
  $G / H \cong \Z \times \Z$ holds, there exists $t,u\in G$ so that $G
  / H = \la tH, uH \ra$ holds. Moreover, as $H$ is finitely generated,
  there exist $a_1,\ldots,a_m\in H$ so that $H = \la a_1,\ldots,a_m \ra$
  holds. Then $G = \la a_1,\ldots, a_m, t, u \ra$ holds and there exists
  a finite set of relations $\R$ with $G \cong \la\,\{a_1,\ldots, a_m,
  t, u \}\mid \R\,\ra$. We choose as Schreier transversal ${\mc T} = \{
  t^i u^j \mid i,j \in \Z \}$. Then, by Lemma~\ref{lem:FirstObserv},
  the subgroup $H$ is invariantly $L$-presented by $\la {\mc Y}
  \mid \emptyset \mid \{ \delta_u, \delta_U, \delta_t, \delta_T \}
  \mid \R^\tau \ra$ where $\delta_x$ denotes the endomorphism of the
  free group $F({\mc Y})$ that is induced by conjugation with $x \in
  \{u,U=u^{-1},t,T=t^{-1}\}$, $\tau$ denotes the Reidemeister rewriting,
  and ${\mc Y} = \{ a_{\ell,i,j}, t_{l,k} \mid i,j,k,l \in \Z, k\neq 0\}$
  are the following Schreier generators:
  \[
    \begin{array}{rcccl}
    a_{\ell,i,j}&=&\gamma(t^i u^j, a_\ell ) &=& t^i u^j a_\ell u^{-j} t^{-i},\\
    t_{i,j}&=&\gamma( t^i u^j, t ) &=& t^i u^j t u^{-j} t^{-1} t^{-i},\\
    u_{i,j}&=&\gamma( t^i u^j, u )  &=& t^i u^j u u^{-j} u^{-1} t^{-i}.
    \end{array}
  \]
  Note that $t_{i,j} = 1$ if and only if $j = 0$ while $u_{i,j} = 1$
  for each $i,j\in\Z$. The endomorphisms $\delta_t$ and $\delta_T$ are
  induced by the maps
  \[
    \delta_t \colon \left\{\begin{array}{rcll}
    a_{\ell,i,j} &\mapsto& a_{\ell,i-1,j},\\
    t_{i,j} &\mapsto& t_{i-1,j},
    \end{array}\right.
    \quad\textrm{and}\quad
    \delta_T \colon \left\{\begin{array}{rcl}
    a_{\ell,i,j} &\mapsto& a_{\ell,i+1,j},\\
    t_{i,j} &\mapsto& t_{i+1,j}, 
    \end{array}\right.
  \]
  for each $i,j\in\Z$; while $\delta_u$ and $\delta_U$ are induced by
  the maps
  \[
    \delta_u \colon \left\{\begin{array}{rcll}
    a_{\ell,i,j}  &\mapsto& (a_{\ell,i,j-1})^{t_{i-1,-1}^{-1} \cdots t_{0,-1}^{-1}},& i\geq 0, j\in\Z,\\[0.75ex]
    a_{\ell,-i,j} &\mapsto& (a_{\ell,-i,j-1})^{t_{-i,-1}\cdots t_{-1,-1}},& i\geq  0, j\in\Z,\\[0.75ex]
    t_{i,j}       &\mapsto& (t_{i,j-1}\,t_{i,-1}^{-1})^{t_{i-1,-1}^{-1} \cdots t_{0,-1}^{-1}}, & i\geq  0, j\in\Z,\\[0.75ex]
    t_{-i,j}      &\mapsto& (t_{-i,j-1}\,t_{-i,-1}^{-1})^{t_{-i,-1}\cdots t_{-1,-1}}, & i\geq  0, j\in\Z,
    \end{array}\right.
  \]
  and
  \[
    \delta_U \colon \left\{\begin{array}{rcll}
    a_{\ell,i,j}  &\mapsto& (a_{\ell,i,j+1})^{t_{i-1,1}^{-1} \cdots t_{0,1}^{-1}},& i\geq  0, j\in\Z,\\[0.75ex]
    a_{\ell,-i,j} &\mapsto& (a_{\ell,-i,j+1})^{t_{-i,1}\cdots t_{-1,1}},& i\geq  0, j\in\Z,\\[0.75ex]
    t_{i,j}       &\mapsto& (t_{i,j+1}\,t_{i,1}^{-1})^{t_{i-1,1}^{-1} \cdots t_{0,1}^{-1}}, & i\geq  0, j\in\Z,\\[0.75ex]
    t_{-i,j}      &\mapsto& (t_{-i,j+1}\,t_{-i,1}^{-1})^{t_{-i,1}\cdots t_{-1,1}}, & i\geq  0, j\in\Z.
    \end{array}\right.
  \]
  We will construct an invariant finite $L$-presentation for the
  subgroup $H$ with generators ${\mc Z} = \{a_1,\ldots,a_m\}
  \cup \{t_1\}$. Define an embedding $\chi\colon
  F({\mc Z}) \to F({\mc Y})$ that is induced by the map
  \[
    \chi\colon \left\{\begin{array}{rcll}
    a_{\ell} &\mapsto& a_{\ell,0,0}, &\textrm{for each }1\leq\ell\leq m\\
    t_1 &\mapsto& t_{0,1}.& 
    \end{array}\right.
  \]
  Write $\Phi = \{\delta_t,\delta_T,\delta_u,\delta_U\}$. For $y
  \in {\mc Z}$ and $\delta \in \Phi$, choose $y^{\chi\delta\gamma}\in
  F({\mc Z})$ with
  \begin{equation}\Label{eqn:ZxZ01}
    y^{-\chi\delta} (y^{\chi\delta\gamma})^\chi \in \ker(\pi).
  \end{equation}
  Define $\iota\colon F({\mc Z}) \to H$ by $\iota = \chi\pi$ where $\pi$
  denotes the free presentation $\pi\colon F({\mc Y}) \to H$ that is given
  by $H$'s invariant $L$-presentation above. For each $\delta\in\Phi$,
  define an endomorphism $\ti\delta\colon F({\mc Z}) \to F({\mc Z})$
  that is induced by the map $y \mapsto y^{\chi\delta\gamma}$.  Then,
  for each $\delta \in \Phi$ and $y \in {\mc Z}$, we obtain
  \[
    y^{\iota\widehat\delta}
    = y^{\chi\pi\widehat\delta}
    = y^{\chi\delta\pi}
    = (y^{\chi\delta})^\pi
    = (y^{\chi\delta\gamma\chi})^\pi
    = (y^{\chi\delta\gamma})^{\chi\pi}
    = y^{\ti\delta\iota}
  \]
  and thus $\ti\delta\iota = \gamma\widehat\delta$. Write $\X =
  \{a_1,\ldots,a_m,t,u\}$ and consider the following subgroups
  of the free group $F({\mc X})$:  Let $E = \la a_1,\ldots,a_m \ra$
  and $S = \la t,u \ra$ be finitely generated subgroups of
  $F({\mc X})$.  Furthermore, let $K \unlhd F({\mc X})$ be the kernel of
  $G$'s free presentation.  Then $G \cong F({\mc X}) / K$ and $H \cong
  E\!K / K$. Moreover, the normal subgroup $E\!K \unlhd F({\mc X})$ is
  supplemented by the finitely generated free group $S$; i.e., $F({\mc
  X}) = S\,E\!K$ holds. Thus $G / H \cong F({\mc X}) / E\!K \cong S /
  S \cap E\!K$. Since $G / H$ is finitely presented, the free subgroup
  $S \cap E\!K$ is finitely generated as a normal subgroup.  The Schreier
  generators ${\mc Y}$ yield that the subgroups
  \[
    E\!K = \la {\mc Y} \ra 
    \quad\textrm{and}\quad
    S\cap E\!K = \la t_{i,j} \mid i,j\in\Z, j \neq 0\ra
  \]
  are freely generated. Moreover, we have that
  \begin{eqnarray*}
    S \cap E\!K &=& \la t_{i,j+1} t_{i,j}^{-1} \mid i,j \in \Z \ra \\
    &=& \la \ldots t_{i,-2} t_{i,-3}^{-1},\: t_{i,-1}t_{i,-2}^{-1},\: 
    t_{i,-1}^{-1},\: t_{i,1},\: t_{i,2} t_{i,1}^{-1},\: t_{i,3} t_{i,2}^{-1},\ldots \mid i\in\Z \ra.
  \end{eqnarray*}
  The latter subgroup is freely generated as the homomorphism $\psi$ that 
  is induced by the map
  \[
   \psi\colon S\cap E\!K \to S\cap E\!K,\: \left\{ \begin{array}{rcll}
   t_{i,j} &\mapsto& t_{i,j+1} t_{i,j}^{-1}, &j < -1 \\
   t_{i,-1}&\mapsto& t_{i,-1}^{-1}, & \\
   t_{i,1} &\mapsto& t_{i,1} &\\
   t_{i,j} &\mapsto& t_{i,j} t_{i,j-1}^{-1}, &j>1
   \end{array}\right.
  \]
  is an automorphism of $S \cap E\!K$ whose inverse is induced by the map
  \[
   \psi^{-1}\colon S\cap E\!K \to S\cap E\!K,\: \left\{ \begin{array}{rcll}
   t_{i,j} &\mapsto& t_{i,j}^{-1} t_{i,j+1}^{-1} \cdots t_{i,-1}^{-1}, &j < -1 \\
   t_{i,-1}&\mapsto& t_{i,-1}^{-1}, & \\
   t_{i,1} &\mapsto& t_{i,1} &\\
   t_{i,j} &\mapsto& t_{i,j} t_{i,j-1} \cdots t_{i,1}, &j>1.
   \end{array}\right.
  \]
  Note that we have 
  \begin{eqnarray*}
    t_{i,j+1} t_{i,j}^{-1} 
   &=&t^i u^{j+1} t u^{-j-1} t^{-1} t^{-i} \cdot 
      (t^i u^j t u^{-j} t^{-1} t^{-i} )^{-1}
    = (t_{0,1})^{u^j t^i}.
  \end{eqnarray*}
  In fact, every element in $S\cap E\!K$ has a unique representation as
  a word in the basis $\{ t^i u^j\cdot t_{0,1}\cdot u^{-j} t^{-i} \mid
  i,j\in\Z \}$ where $t_{0,1} = utUT$ is a normal generator of $S \cap
  E\!K = \la t_{0,1} \ra^S$. More precisely, for $i\geq 0$ and $j>0$,
  we have the following representatives in free subgroup $S \cap E\!K
  \leq F({\mc Y})$:
  \[
   \begin{array}{rcl}
   t_{i,j} &=& \left(t_{0,1}^{\delta_U^{j-1}} \cdot t_{0,1}^{\delta_U^{j-2}}
   \cdots t_{0,1} \right)^{\delta_T^i}\\
   t_{-i,j} &=& \left(t_{0,1}^{\delta_U^{j-1}} \cdot t_{0,1}^{\delta_U^{j-2}}
   \cdots t_{0,1} \right)^{\delta_t^i}
   \end{array}
   \quad\textrm{and}\quad
   \begin{array}{rcl}
   t_{i,-j}
   &=& \left(t_{0,1}^{-\delta_u^{j}} \cdot t_{0,1}^{-\delta_u^{j-1}} 
   \cdots t_{0,1}^{-\delta_u}\right)^{\delta_T^i} \\
   t_{-i,-j}
   &=& \left(t_{0,1}^{-\delta_u^{j}} \cdot t_{0,1}^{-\delta_u^{j-1}} 
   \cdots t_{0,1}^{-\delta_u}\right)^{\delta_t^i}.
   \end{array}
  \]
  The Schreier generators $a_{\ell,i,j}$ are conjugates
  of the generators $a_{\ell,0,0}$ so that
  \[
    \begin{array}{rcl}
    a_{\ell,i,j} &=& (a_{\ell,0,0})^{\delta_U^j\delta_T^i} \\[0.50ex]
    a_{\ell,-i,j} &=& (a_{\ell,0,0})^{\delta_U^j\delta_t^i}
    \end{array}
    \quad\textrm{and}\quad
    \begin{array}{rcl}
    a_{\ell,i,-j} &=& (a_{\ell,0,0})^{\delta_u^j\delta_T^i} \\[0.50ex]
    a_{\ell,-i,-j} &=& (a_{\ell,0,0})^{\delta_u^j\delta_t^i}.
    \end{array}
  \]
  In particular, we can choose the following basis $\widehat{\mc Y}$ for 
  the free subgroup $E\!K$:
  \[
    \widehat{\mc Y} 
    =  \left\{ (a_{\ell,0,0})^{\delta_U^j\delta_T^i},
      \ldots (a_{\ell,0,0})^{\delta_u^j \delta_t^i},
      (t_{0,1})^{\delta_U^j\delta_T^i},\ldots,
      (t_{0,1})^{\delta_u^j\delta_t^i} \right\}_{ i,j\geq 0}.
  \]
  Define $\gamma\colon F(\widehat{\mc Y}) \to F({\mc Z})$ to be induced
  by the map
  \[
    \gamma\colon \left\{ \begin{array}{rcl}
    (a_{\ell,0,0})^{\delta_U^j\delta_T^i}   &\mapsto&(a_{\ell})^{\ti\delta_U^j \ti\delta_T^i}, \\[0.5ex]
    (a_{\ell,0,0})^{\delta_U^j\delta_t^i}  &\mapsto&(a_{\ell})^{\ti\delta_U^j \ti\delta_t^i}, \\[0.5ex]
    (a_{\ell,0,0})^{\delta_u^j\delta_T^i}  &\mapsto&(a_{\ell})^{\ti\delta_u^j \ti\delta_T^i}, \\[0.5ex]
    (a_{\ell,0,0})^{\delta_u^j\delta_t^i} &\mapsto&(a_{\ell})^{\ti\delta_u^j \ti\delta_t^i}, 
    \end{array}\right.
    \quad\textrm{and}\quad
    \gamma\colon \left\{ \begin{array}{rcl}
    (t_{0,1})^{\delta_U^j\delta_T^i}   &\mapsto& (t_{1})^{\ti\delta_U^j\ti\delta_T^i}  \\[0.5ex]
    (t_{0,1})^{\delta_U^j\delta_t^i}   &\mapsto& (t_{1})^{\ti\delta_U^j\ti\delta_t^i}  \\[0.5ex]
    (t_{0,1})^{\delta_u^j\delta_T^i}   &\mapsto& (t_{1})^{\ti\delta_u^j\ti\delta_T^i}   \\[0.5ex]
    (t_{0,1})^{\delta_u^j\delta_t^i}   &\mapsto& (t_{1})^{\ti\delta_u^j\ti\delta_t^i},
    \end{array}\right.
  \]
  where $i,j\geq 0$. Then $\gamma$ acts on the Schreier
  generators ${\mc Y}$ as follows:
  \[
    \gamma\colon \left\{ \begin{array}{rcl}
    a_{\ell,i,j}   &\mapsto&(a_{\ell})^{\ti\delta_U^j \ti\delta_T^i}, \\[0.75ex]
    a_{\ell,-i,j}  &\mapsto&(a_{\ell})^{\ti\delta_U^j \ti\delta_t^i}, \\[0.75ex]
    a_{\ell,i,-j}  &\mapsto&(a_{\ell})^{\ti\delta_u^j \ti\delta_T^i}, \\[0.75ex]
    a_{\ell,-i,-j} &\mapsto&(a_{\ell})^{\ti\delta_u^j \ti\delta_t^i}, 
    \end{array}\right.
    \quad\textrm{and}\quad
    \gamma\colon \left\{ \begin{array}{rcl}
    t_{i,j}   &\mapsto& (t_{1}^{\ti\delta_U^{j-1}} \cdots t_1)^{\ti\delta_T^i}, \\[0.75ex]
    t_{-i,j}  &\mapsto& (t_{1}^{\ti\delta_U^{j-1}} \cdots t_1)^{\ti\delta_t^i}, \\[0.75ex]
    t_{i,-j}  &\mapsto&
    (t_1^{-\ti\delta_u^j} \cdots t_1^{-\ti\delta_u})^{\ti\delta_T^i},\\[0.75ex]
    t_{-i,-j} &\mapsto&
    (t_{1}^{-\ti\delta_u^{j}} \cdots t_{1}^{-\ti\delta_u})^{\ti\delta_t^i},
    \end{array}\right.
  \]
  where $i \geq 0$ and $j>0$. For $i\geq 0$ and $j >0$, the element
  $t_{i,j} \in {\mc Y}$ is mapped by $\gamma\iota$ to
  \begin{eqnarray*}
    t_{i,j}^{\gamma\iota} 
    &=&(t_{1}^{\ti\delta_U^{j-1}} \cdots t_1)^{\ti\delta_T^i\iota} 
     = (t_{1}^{\ti\delta_U^{j-1}\ti\delta_T^i} \cdots t_1^{\ti\delta_T^i})^\iota
     = t_{1}^{\iota\widehat\delta_U^{j-1}\widehat\delta_T^i} \cdots t_1^{\iota\widehat\delta_T^i}\\
    &=&(t_{1}^{\chi\pi\widehat\delta_U^{j-1}\widehat\delta_T^i} \cdots t_1^{\chi\pi\widehat\delta_T^i})
     = (t_{0,1}^{\delta_U^{j-1}\delta_T^i} \cdots t_{0,1}^{\delta_T^i})^\pi 
     = (t_{i,j})^\pi
  \end{eqnarray*}
  because $\ti\delta \iota = \iota \widehat\delta$ holds. Similarly, we
  obtain that $a_{\ell,i,j}^{\gamma\iota} = a_{\ell,i,j}^\pi$ holds. Thus
  $\gamma\iota = \pi$. Define the normal subgroup
  \[
    N = \Big\la \bigcup_{\sigma \in \ti\Phi^*} \Big(
    \big\{
    (y^{-1} y^{\gamma\chi})^{\delta \gamma} \big\}_{
    y \in {\mc Y} \setminus {\mc Z}, \delta \in \Phi }\Big)^\sigma
    \Big\ra^{F({\mc Z})}.
  \]
  We prove that $F({\mc Z}) / N $ is invariantly finitely $L$-presented
  so that Lemma~\ref{lem:FZNFac} applies. For $i \geq 0$ and $j>0$,
  it holds that
  \[
    \begin{array}{c@{\:=\:}c@{\:=\:}c@{\:=\:}c}
    (t_{i,j}^{-1} t_{i,j}^{\gamma\chi})^{\delta_T\gamma} 
    & t_{i+1,j}^{-\gamma} t_{i,j}^{\gamma\ti\delta_T} 
    & (t_1^{\ti\delta_U^{j-1}} \cdots t_1)^{-\ti\delta_T^{i+1}}
        (t_1^{\ti\delta_U^{j-1}} \cdots t_1)^{\ti\delta_T^i \, \ti\delta_T}
    &1,\\
    (t_{-i,j}^{-1} t_{-i,j}^{\gamma\chi})^{\delta_t\gamma}
    & t_{-i-1,j}^{-\gamma} t_{-i,j}^{\gamma\ti\delta_t}
    & (t_1^{\ti\delta_U^{j-1}} \cdots t_1)^{-\ti\delta_t^{i+1}}
        (t_1^{\ti\delta_U^{j-1}} \cdots t_1)^{\ti\delta_t^i \, \ti\delta_t}
    &1,\\
    (t_{i,-j}^{-1} t_{i,-j}^{\gamma\chi})^{\delta_T\gamma} 
    & t_{i+1,-j}^{-\gamma} t_{i,-j}^{\gamma\ti\delta_T} 
    & (t_1^{-\ti\delta_u^{j}} \cdots t_1^{-\ti\delta_u})^{-\ti\delta_T^{i+1}}
        (t_1^{-\ti\delta_u^{j}} \cdots t_1^{-\ti\delta_u})^{\ti\delta_T^i \, \ti\delta_T}
    &1,\\
    (t_{-i,-j}^{-1} t_{-i,-j}^{\gamma\chi})^{\delta_t\gamma}
    & t_{-i-1,-j}^{-\gamma} t_{-i,-j}^{\gamma\ti\delta_t}
    & (t_1^{-\ti\delta_u^{j}} \cdots t_1^{-\ti\delta_u})^{-\ti\delta_t^{i+1}}
        (t_1^{-\ti\delta_u^{j}} \cdots t_1^{-\ti\delta_u})^{\ti\delta_t^i \, \ti\delta_t} 
    &1.
    \end{array}
  \] 
  For $i = 0$, we also have that
  \[
    \begin{array}{c@{\:=\:}c@{\:=\:}c@{\:=\:}c}
    (t_{0,j}^{-1} t_{0,j}^{\gamma\chi})^{\delta_U\gamma}
    & (t_{0,j+1}\,t_{0,1}^{-1})^{-\gamma} t_{0,j}^{\gamma\ti\delta_U} 
    & (t_1^{\ti\delta_U^j} \cdots t_1^{\ti\delta_U}\cdot t_1\cdot t_1^{-1} )^{-1} 
        (t_1^{\ti\delta_U^{j-1}} \cdots t_1 )^{\ti\delta_U}
    & 1,\\
    (t_{0,-j}^{-1} t_{0,-j}^{\gamma\chi})^{\delta_u\gamma}
    & (t_{0,-j-1}\,t_{0,-1}^{-1})^{-\gamma} t_{0,-j}^{\gamma\ti\delta_u} 
    & (t_1^{-\ti\delta_u^{j+1}} \cdots t_1^{-\ti\delta_u}\cdot t_1^{\ti\delta_u} )^{-1} 
        (t_1^{-\ti\delta_u^{j}} \cdots t_1^{\ti\delta_u})^{\ti\delta_u}
    & 1.
    \end{array}
  \]
  However, we also need to consider the image $t_{i,j}^{\delta_U} =
  t_{0,1}^{\delta_T^i\delta_U}$ with $i>0$. Notice that in the 
  finitely presented monoid $S / S \cap E\!K$ the following holds:
  \[
     \begin{array}{rclclcl}
     TU 
     &=& UT \cdot tuTU 
     &=& UT \cdot (utUT)^{-1}
     &=& UT \cdot t_{0,1}^{-1},\\
     Tu
     &=& uT \cdot tUTu 
     &=& uT \cdot (utUT)^{\delta_u}
     &=& uT \cdot t_{0,1}^{\delta_u},\\
     tU 
     &=& U\cdot (utUT) \cdot t 
     &=& U\cdot t_{0,1} \cdot t
     &=& Ut \cdot t_{0,1}^{\delta_t},\\
     tu 
     &=& u\cdot (UtuT) \cdot t 
     &=& u \cdot t_{0,-1} \cdot t
     &=& ut \cdot t_{0,1}^{-\delta_u\delta_t}.
     \end{array}
  \]
  Denote by $\Delta(x) \colon F({\mc Y}) \to F({\mc Y}), \:g\mapsto
  x^{-1} g x$ the inner automorphism of $F({\mc Y})$ that is induced
  by conjugation with $x \in F({\mc Y})$. Then $\delta\in \Phi =
  \{\delta_u,\delta_U,\delta_t,\delta_T\}$ satisfy
  \[
   \begin{array}{rcl}
   \delta_T \delta_U&=&\delta_U \delta_T \cdot \Delta( t_{0,1}^{-1} ),\\[0.5ex]
   \delta_T \delta_u&=&\delta_u \delta_T \cdot \Delta( t_{0,1}^{\delta_u}),
   \end{array}
   \quad\textrm{and}\quad
   \begin{array}{rcl}
   \delta_t \delta_U&=&\delta_U \delta_t \cdot \Delta( t_{0,1}^{\delta_t}),\\[0.5ex]
   \delta_t \delta_u&=&\delta_u \delta_t \cdot \Delta( t_{0,1}^{-\delta_u\delta_t}).
   \end{array}
  \]
  We prove that $F({\mc Z}) / N$ is invariantly finitely
  $L$-presented by $\la \{a_1,\ldots,a_m,t_1\} \mid \emptyset \mid \ti\Phi
  \mid {\mc V} \ra$ where the iterated relations in ${\mc V}$ are given by
  \[
    {\mc V} = \Big\{ y^{-1} y^{\ti\delta_t\ti\delta_T}, \ldots, y^{-1} y^{\ti\delta_U\ti\delta_u}, y^{-\ti\delta_T\ti\delta_U} y^{\ti\delta_U\ti\delta_T \Delta( t_1^{-1} ) },
     \ldots, y^{-\ti\delta_t\ti\delta_u} y^{\ti\delta_u\ti\delta_t \Delta(t_1^{-\ti\delta_u\ti\delta_t})} \Big\}_{ y \in {\mc Z} }
  \]
  that is, we prove that $M = \la \bigcup_{\ti\sigma\in\ti\Phi} {\mc
  V}^{\ti\sigma} \ra^{F({\mc Z})}$ and $N$ coincide. We first note that
  \begin{eqnarray*}
   N \ni (t_{1,1}^{-1}\,t_{1,1}^{\gamma\chi})^{\delta_U\gamma}
    &=& t_{1,1}^{-\delta_U\gamma} \, t_{1,1}^{\gamma\ti\delta_U}
     =  t_{0,1}^{-\delta_T\delta_U\gamma} \, t_{1}^{\ti\delta_T\ti\delta_U}
     =  t_{0,1}^{-\delta_U\delta_T \cdot \Delta(t_{0,1}^{-1}) \gamma} \, t_{1}^{\ti\delta_T\ti\delta_U} \\
    &=& t_{0,1}^{-\delta_U\delta_T \gamma \cdot \Delta(t_{1}^{-1}) } \, t_{1}^{\ti\delta_T\ti\delta_U} 
     =  t_{1}^{-\ti\delta_U\ti\delta_T  \cdot \Delta(t_{1}^{-1}) } \, t_{1}^{\ti\delta_T\ti\delta_U} \in {\mc V}.
  \end{eqnarray*}
  Similar computations show that the elements in ${\mc V}$ appear a among
  the normal generators of $N$. Thus $M \subseteq N$. On the other hand,
  for $i>0$ and $j>0$, we have that
  \begin{eqnarray*}
     t_{i,j}^{\delta_U\gamma} 
     &=& (t_{0,1}^{\delta_U^{j-1}} \cdots t_{0,1})^{\delta_T^i\delta_U \gamma}\\
     &=& (t_{0,1}^{\delta_U^{j-1}} \cdots t_{0,1})^{\delta_T^{i-1}\delta_U\delta_T\cdot \Delta(t_{0,1}^{-1}) \gamma} \\
     &=& \ldots = (t_{0,1}^{\delta_U^{j-1}} \cdots t_{0,1})^{\delta_U\delta_T^{i} \cdot \Delta(t_{0,1}^{-\delta_T^{i-1}}) \Delta(t_{0,1}^{-\delta_T^{i-2}})\cdots  \Delta(t_{0,1}^{-1}) \gamma} \\
     &=& (t_{0,1}^{\delta_U^{j-1}} \cdots t_{0,1})^{\delta_U\delta_T^{i} \gamma \cdot \Delta(t_{1}^{-\ti\delta_T^{i-1}})  \Delta(t_{1}^{-\ti\delta_T^{i-2}})\cdots  \Delta(t_{1}^{-1}) } \\
     &=& (t_{0,1}^{\delta_U^{j}\delta_T^i\gamma} \cdots t_{0,1}^{\delta_U\delta_T^{i} \gamma})^{\Delta(t_{1}^{-\ti\delta_T^{i-1}})  \Delta(t_{1}^{-\ti\delta_T^{i-2}})\cdots  \Delta(t_{1}^{-1}) } \\
     &=& (t_{1}^{\ti\delta_U^{j-1}} \cdots t_{1})^{\ti\delta_U\ti\delta_T^{i} \cdot \Delta(t_{1}^{-\ti\delta_T^{i-1}}) \Delta(t_{1}^{-\ti\delta_T^{i-2}})\cdots  \Delta(t_{1}^{-1}) } \\
     &=& (t_{1}^{\ti\delta_U^{j-1}} \cdots t_{1})^{{\ti\delta_U\ti\delta_T}\Delta(t_{1}^{-1}) \ti\delta_T^{i-1}\cdot \Delta(t_1^{-\ti\delta_T^{i-2}})\cdots \Delta(t_{1}^{-1})} \\
     &\equiv& (t_{1}^{\ti\delta_U^{j-1}} \cdots t_{1})^{{\ti\delta_T\ti\delta_U}\ti\delta_T^{i-1}\cdot \Delta(t_1^{-\ti\delta_T^{i-2}})\cdots \Delta(t_{1}^{-1})}\mod M\\ 
     &\equiv& \ldots \equiv (t_{1}^{\ti\delta_U^{j-1}} \cdots t_{1})^{{\ti\delta_T^i\ti\delta_U}} \mod M 
  \end{eqnarray*}
  and $(t_{1}^{\ti\delta_U^{j-1}} \cdots
  t_{1})^{{\ti\delta_T^i\ti\delta_U}} =
  (t_{i,j})^{\gamma\ti\delta_U}$. Thus $(t_{i,j}^{-1}
  t_{i,j}^{\gamma\chi})^{\delta_U\gamma} \in M$. It follows analogously
  for the other normal generators of $N$ that these are contained
  in $M$. Thus $F({\mc Z}) / N$ is invariantly finitely $L$-presented and 
  so is our subgroup $H$.
\end{proof}
By~\cite[Theorem~6.1]{Har11b}, every finite index subgroup
$H$ of an invariantly finitely $L$-presented group $G =
\la\X\mid\Q\mid\Phi\mid\R\ra$ is invariantly finitely $L$-presented
whenever the substitutions $\sigma\in\Phi$ induce endomorphisms of the
subgroup $H$. This allows us to prove Theorem~\ref{thm:TorRk2} using the
results in Theorem~\ref{thm:Rank2} and Theorem~\ref{thm:Ben11}:
\def\0{ of Theorem~\ref{thm:TorRk2}}
\begin{CenProof}{\0}
  Let $G$ be a finitely presented group and let $H\unlhd G$ be a finitely
  generated normal subgroup so that $G / H$ is abelian with torsion-free rank at most
  two.  Since $G$ is finitely generated, $G / H$ is a finitely generated
  abelian group and so it decomposes into $G / H \cong \Z^\ell \times T$
  with torsion subgroup $T$ and torsion-free rank $\ell \leq 2$. Denote
  by $U \leq G$ the full preimage of the torsion subgroup $T$ in $G$.
  Then $G / U \cong \Z^\ell$ and $[U:H] < \infty$ hold. If $\ell = 0$
  holds, $H$ has finite index in $G$ and thus it is invariantly finitely
  $L$-presented by~\cite[Theorem~6.1]{Har11b}. If either $\ell = 1$ or
  $\ell = 2$ holds, the subgroup $U \unlhd G$ is invariantly finitely
  $L$-presented by Theorem~\ref{thm:Ben11} or Theorem~\ref{thm:Rank2}.
  Each substitution in the $L$-presentation of $U$ is induced by
  conjugation within the finitely presented group $G$. Since $H$ is a
  normal subgroup of $G$ each substitution of the finite $L$-presentation
  of $U$ stabilizes the subgroup $H$. Thus~\cite[Theorem~6.1]{Har11b}
  applies to the finite index subgroup $H \unlhd U$ and it shows that $H$
  is invariantly finitely $L$-presented.
\end{CenProof}
\noindent In the proof of Theorem~\ref{thm:Rank2}, it is essential that
the elements $g \in S \cap E\!K$ have a unique representation in the basis
$\{t^i s^j \cdot t_{0,1}\cdot s^{-j} t^{-i} \mid i,j\in\Z\}$. This allows
us to define the epimorphism $\gamma\colon F({\mc Y}) \to F({\mc Z})$
so that it maps conjugates by elements of the Schreier transversal to
images of automorphisms which are induced by conjugation with a Schreier
transversal. Since $S / S \cap E\!K$ is finitely presented, we can always
choose finitely many Schreier generators ${\mc W} \subseteq {\mc Y}$ so
that $S \cap E\!K$ is generated, as a normal subgroup, by ${\mc W}$. In
our proof of Theorem~\ref{thm:Rank2} the conjugates of these elements in
${\mc W}$ by elements of the Schreier transversal from a basis for the
subgroup $S \cap E\!K$.  This is no longer possible for $G / H \cong \Z
\times \Z \times \Z$:
\begin{remark}\Label{rem:ZxZxZ}
  Consider the notation from the proof of Theorem~\ref{thm:Rank2}.
  For $G / H \cong \Z \times \Z \times \Z$, we choose as Schreier
  transversal ${\mc T} = \{ r^i s^j t^k \mid i,j,k\in\Z\}$ and we obtain
  the Schreier generators:
  \begin{eqnarray*}
    a_{\ell,i,j,k} = \gamma(r^i s^j t^k, a_\ell ) &=& r^i s^j t^k a_{\ell} t^{-k} s^{-j} r^{-i},\\
    s_{i,j,k} = \gamma(r^i s^j t^k, s) &=& r^i s^j (t^k st^{-k} s^{-1}) s^{-j} r^{-i},\\
    r_{i,j,k} = \gamma(r^i s^j t^k, r) &=& r^i ( s^j t^k r t^{-k} s^{-j} r^{-1} ) r^{-i},\\
    t_{i,j,k} = \gamma(r^i s^j t^k, t) &=& 1,
  \end{eqnarray*}
  where $s_{i,j,k} = 1$ if and only if $k = 0$ while $r_{i,j,k} = 1$ if
  and only if $(j,k) = (0,0)$. Then
  \[
    E\!K = \left\la a_{\ell,i,j,k}, s_{i,j,o}, r_{i,p,q} \mid 1\leq\ell\leq m, i,j,k,o,p,q \in \Z, o\neq 0, (p,q) \neq (0,0) \right\ra
  \]
  is freely generated and so is
  \[
    S \cap E\!K = \la s_{i,j,o}, r_{i,p,q} \mid i,j,o,p,q\in\Z,
    o\neq 0, (p,q) \neq (0,0) \ra.
  \]
  Since $G / H \cong S / S \cap E\!K \cong \Z \times \Z \times \Z$ is 
  finitely presented, 
  the subgroup $S \cap E\!K$ is finitely generated as a normal subgroup 
  of $S$. In particular, we have that 
  \[
    S / S \cap E\!K = \la r,s,t \mid
    \underbrace{tst^{-1}s^{-1}}_{= s_{0,0,1}},
    \underbrace{trt^{-1}r^{-1}}_{= r_{0,0,1}},
    \underbrace{srs^{-1}r^{-1}}_{= r_{0,1,0}} \ra
  \]
  so that $S \cap E\!K = \la s_{0,0,1}, r_{0,0,1}, r_{0,1,0} \ra^S$ holds.
  The normal generators of $S \cap E\!K$ satisfy
  \[
    \begin{array}{rcl}
    r^i s^j t^k \cdot s_{0,0,1}\cdot t^{-k} s^{-j} r^{-i} 
    &=& s_{i,j,k+1}\cdot s_{i,j,k}^{-1},\\[0.5ex]
    r^i s^j t^k \cdot r_{0,0,1}\cdot t^{-k} s^{-j} r^{-i} 
    &=& r_{i,j,k+1}\cdot r_{i,j,k}^{-1},\\[0.5ex]
    r^i s^j t^k \cdot r_{0,1,0}\cdot t^{-k} s^{-j} r^{-i} 
    &=& s_{i,j,k}\cdot r_{i,j+1,k}\cdot s_{i+1,j,k}^{-1}\cdot r_{i,j,k}^{-1}.
    \end{array}
  \]
  It can be seen easily (e.g. using \Gap) that
  \[
    U = \{ s_{i,j,k+1}\,s_{i,j,k}^{-1},\:
           r_{i,j,k+1}\,r_{i,j,k}^{-1},\:
           s_{i,j,k}\,r_{i,j+1,k}\,s_{i+1,j,k}^{-1}\,r_{i,j,k}^{-1} \}_{i,j,k\in\Z}
  \]
  is not a basis for $S \cap E\!K$.  Therefore the ideas in the proof
  of Theorem~\ref{thm:Rank2} do not apply.
\end{remark}

\subsection*{Acknowledgments}
I am grateful to Laurent Bartholdi for valuable comments and
suggestions.

\def\cprime{$'$}

\noindent Ren\'e Hartung,
{\scshape Mathematisches Institut},
{\scshape Georg-August Universit\"at zu G\"ottingen},
{\scshape Bunsenstra\ss e 3--5},
{\scshape 37073 G\"ottingen},
{\scshape Germany}\\[1ex]
{\it Email:} \qquad \verb|rhartung@uni-math.gwdg.de|\\[2.ex]

\end{document}